\documentclass[10pt,twoside]{article}
\usepackage{amssymb}
\usepackage[mathscr]{eucal}
\usepackage{eufrak}
\usepackage{amsmath,amsthm}
\usepackage{mathrsfs}
\usepackage[colorlinks=true,
   urlcolor=blue,           filecolor=green,      
   citecolor=green,      
   linkcolor=red,           bookmarks=true,
  unicode,
   plainpages=false,   ]{hyperref}
\usepackage{color}

\def\n{\nabla}
\def\intl#1{\int\limits_{#1}}
\def\intll#1#2{\int\limits_{#1}^{#2}}
\def\dm{|\hskip-0.05cm|}
\def\OO{\Omega}
\def\displ{\displaystyle}
\def\VSE{\vspace{6pt}\\&\displ }
\def\VS{\vspace{6pt}\\\displ }
\def\rf#1{{\rm(\ref{#1})}}
\def\chiu{\hfill$\displaystyle\vspace{4pt}
\underset\Box\null$\par}
\def\R{\Bbb R}
\def\N{\Bbb N}

\def\à{\`{a}}
\def\vep{\varepsilon}
\def\be{\begin{equation}}
\def\ba{\begin{array}}
\def\ea{\end{array}}
\def\ee{\end{equation}}
 \def\ov{\overline}
\def\po{{\partial\Omega}}
\def\é{\'{e}}
\font\sc=cmcsc10
\title{\Large \bf On the $L^p$-$L^q$ estimates     of the gradient   of  solutions\\ to the Stokes problem}
\author{\sc   Paolo Maremonti
\thanks{Dipartimento di Matematica e Fisica,  
Universit\`{a} degli Studi della Campania
``L. Vanvitelli'', via Vivaldi 43, 81100 Caserta,
 Italy.
paolo.maremonti@unicampania.it\newline\null\quad\; The research  is partially supported by GNFM (INdAM) and by MIUR via the PRIN 2017 {\it ``Hyperbolic Systems of Conservation Laws and Fluid Dynamics: Analysis and Applications''.}}}
\date{}
\begin{document}
\markboth{\footnotesize\rm P.
Maremonti} {\footnotesize\rm
On $L^p$-$L^q$-estimates of the gradient of   solutions to the Stokes problem}
\maketitle \noindent{\bf Abstract} - {\small The paper is concerned with estimates of the gradient of the solutions  to the Stokes IBVP both in a bounded and in an exterior domain. More precisely, we look for estimates of the kind $\dm \n v(t)\dm_q \leq g(t)\dm \n v_0\dm_p\,,\;q\geq p>1\,,$ for all $t>0$ where function $g$ is independent of $v_0$.} 
\vskip 0.2cm
 \par\noindent{\small Keywords:  Stokes problem,   semigroup properties. }
  \par\noindent{\small  
  AMS Subject Classifications: 35B45, 35Q30, 76D07.}  
 \par\noindent
 \vskip -0.7true cm\noindent
\newcommand{\red}{\protect\bf}
\renewcommand\refname{\centerline
{\red {\normalsize \bf References}}}
\newtheorem{ass}
{\bf Assumption} 
\newtheorem{defi}
{\bf Definition} 
\newtheorem{tho}
{\bf Theorem} 
\newtheorem{rem}
{\sc Remark} 
\newtheorem{lemma}
{\bf Lemma} 
\newtheorem{coro}
{\bf Corollary} 
\newtheorem{prop}
{\bf Proposition} 
\renewcommand{\theequation}{\arabic{equation}}
\setcounter{section}{0}
\section{Introduction}\label{intro}
 We consider the Stokes initial boundary value problem in a   domain $\OO\subseteq\R^n$, $n\geq2$, that can be assumed bounded or   exterior, whose boundary $\po$ is supposed to be smooth:
\be\label{STI}\ba{l}v_t-\Delta v=-\nabla
\pi_v,\;\nabla\cdot v=0,\mbox{ in
}(0,T)\times\OO,\vspace{3pt}\\v=0\mbox{
on
}(0,T)\times\po,\vspace{3pt}\\
v=v_0\mbox{ on }\{0\}\times\OO. \ea\ee Several authors (see e.g. \cite{BM-I}-\cite{BS}, \cite{DS-I}-\cite{DHP},\cite{FKS}-\cite{FS},\cite{Gi}-\cite{GS},\cite{I},\cite{MS-II},\cite{U}-\cite{Y}) have contributed to the study of  semigroup properties of the Stokes operator associated to problem \rf{STI}, and of the related $L^p$-$L^q$ estimates of  solutions. In particular,    for $q\in[p,\infty]$, set $\mu:=\frac n2\left(\frac1p-\frac1q\right)$, the following hold:
\be\label{JPAIn} \ba{lll}\dm v(t )\dm
_q\!\leq c\dm v(s)\dm _p(t\!-\!s)^{-\mu},&
 \mbox{for all }t-s>0;\vspace{4pt}\\\dm
\nabla v(t )\dm _q\!\leq c\dm v(s)\dm
_p(t\!-\!s)^{-\mu_1},\hskip-0.2cm&\hskip-0.1cm
\mu_1\hskip-0.1cm:=\!\!\left\{\hskip-0.2cm\ba{ll}\frac12\!+\!\mu\hskip-0.4cm
&\mbox{if }t\!-\!s\in(0,1],\\\frac12\!+\!\mu
&\mbox{if }t\!-\!s>\!0,q\!\in\![p,n],
\\\frac n{2p}&\mbox{if }t\!-\!s>1\,,q\geq n\,;\ea\right.
\vspace{4pt}\\\dm v_t(t )\dm _q\!\leq
c\dm v(s)\dm _p(t\!-\!s)^{-\mu_2},\;
&\hskip-0.1cm\mu_2\hskip-0.1cm:=1+\mu,\mbox{ for all }t-s>0;\vspace{4pt}\\\dm D^2v(t )\dm_q\!\leq c \dm v(s)\dm_p(t\!-\!s)^{-\mu_3}, \hskip-0.1cm&\hskip-0.1cm\mu_3\hskip-0.1cm:=\!\!\left\{\hskip-0.2cm\ba{ll}1\!+\!\mu\hskip-0.1cm&\mbox{if }t\!-\!s\in(0,1],\\1\!+\!\mu& \mbox{if }t\!-\!s\!>0,q\!\in\![p,\hskip-0.05cm\frac n2],
\\\frac{n}{2p}&\mbox{if }t\!-\!s>1, q\geq\frac n2\,;\ea\right.
\ea\ee
where the constant $c$ is
independent of $v$ and, in a suitable sense,  the
exponents $\mu, \,\mu_i,\,i=1$-$3$, are sharp (see Lemma\,\ref{JPA} below and related references). More recently, also   the case of the initial data in $L^\infty(\OO)$ has   been considered   by some authors (see \cite{AG-I}-\cite{BH},\cite{CC},\cite{HM},\cite{M},\cite{SMM}-\cite{Sll}). In particular for $n\ge3$ the following estimates hold: \be\label{PE}\ba{l}\dm v(t)\dm_\infty\leq c\dm v(s)\dm_\infty\,,\quad t-s>0\,,\VS  \dm \n v(t)\dm_\infty\leq c\dm v(s)\dm_\infty (t-s+1)^\frac12(t-s)^{-\frac12}\,,\quad t-s>0\,,\VS\dm v_t(t)\dm_\infty\leq c\dm v(s)\dm_\infty(t-s)^{-1}\,,\quad t-s>0\,,\ea\ee where the constant $c$ is independent of $v_0$ and again the estimate  \rf{PE}$_2$ for $\n v$ is sharp (see \cite{HM}).\par The aim of this paper is to study        $L^p$-$L^q$-norm of the gradient of the solutions, that is, we look for estimates of the kind, $q\geq p$ and $p\in (1,\infty)$,\be\label{GI}\dm \n v(t)\dm_q\leq g(t)\dm \n v_0\dm_p\,, \mbox{ for all }t>0\,,\ee where $g(t)$ is independent of $v_0$.
As far as we know, the literature related to the previous question is not wide. In the case of the Stokes operator, for any domain which is sufficiently regular, estimate \rf{GI} holds for $p=2$ (see  e.g. \cite{Hy-I,Hy-II}). Moreover, for all $p\in(1,\infty)$, making use of the representation formula of the solutions, estimate \rf{GI} holds in the case of solutions to the Cauchy problem and of the IBVP in the half-space \cite{MSh-II}, and recently, in the interesting paper \cite{HK}, the result is achieved for $p=\infty$ (see Proposition\,3.2). Even the heat equation  has only few results.  In \cite{En,EnKL}, for solutions to the $\frak p$-parabolic equation, the authors obtain    some special results which are related to some bounded domains. More precisely, in \cite{En} the author  considers the heat equation (that is $\frak p=2$), with homogenous Dirichlet boundary condition  or homogeneous Nuemann boundary condition, and  proves that the function $e^{\lambda_pt}\dm \n u(t)\dm_p$ is non increasing, for all $p\in(1,\infty)$. The constant $\lambda_p$ is the minimal eigenvalue of a suitable boundary value problem associated to $-\Delta$, where $\lambda_p$ can be negative (e.g. if $\OO\subset\R^2$ is multiconnected). Finally,   estimate \rf{GI}   is proved with $g(t)=c$ in the case of $q=p\in [\ov p/(\ov p-1),\ov p]$, for a suitable $\ov p>2$. In the paper  \cite{EnKL} there is an extension  of the results proved in \cite{En} to the solutions to the $\frak p$-parabolic heat equation.\par Before going into   the results of this paper, we point out that, beyond the intrinsic interest related to the Stokes problem,   the paper is motivated by the fact that the results allow  us to extend   to  the three-dimensional initial boundary value problem some results of the ones obtained in \cite{MSh-I, MSh-II} for the 2D-Navier-Stokes, in particular furnishing weak solutions for non decaying data.    \par In order to state our chief result we introduce the set of the hydrodynamic test functions $\mathscr C_0(\OO)$ and, for $p\in (1,\infty)$, $J^p_0(\OO):=$completion of $\mathscr C_0(\OO)$ with respect to the seminorm $\dm\n\cdot\dm_p$ (norm for $p\in(1,n))$.    
\par We are able to prove \begin{tho}\label{CT}{\sl Let $\OO$ be a bounded domain.  Let $v_0\in J_0^p(\OO)$. Then there exists a unique solution to problem \rf{STI} such that, for all $T>0$, \be\label{CT-oi}v\in C([0,T;J^p_0(\OO)),\; t^\frac12v_t,\, t^\frac12D^2v,\,t^\frac12 \n\pi_v\in L^\infty(0,T;L^p(\OO))\,.\ee For   $q\geq p$, set $\mu:=\frac n2\left(\frac1p-\frac1q\right)$,   the following   hold with a constant $c$ independent of $v$:
\be\label{CT-ii}\ba{l}\dm \n v(t)\dm_q\!\leq  c (t-\sigma)^{-\mu}\!\exp[-\gamma (t\!-\!\sigma)]\,\dm \n v(\sigma)\dm_p\,,\mbox{\,for\,all } t>\!\sigma\!\geq0\,,\VS \dm v_t(t)\dm_q\leq c (t-\sigma)^{-\frac12-\mu}\!\exp[-\gamma (t\!-\!\sigma)]\,\dm \n v(\sigma)\dm_p\,, \mbox{\,for\,all } t>\!\sigma\!\geq0\,.\ea\ee}\end{tho}
 \begin{tho}\label{CTE}{\sl Let $\OO$ be an   exterior domain.  Let $v_0\in J_0^p(\OO)$. Then there exists a unique solution to problem \rf{STI} such that, for all $T>0$, \be\label{CT-i}v\in C([0,T;J^p_0(\OO)),\; t^\frac12v_t,\, t^\frac12D^2v,\,t^\frac12\nabla\pi_v\in L^\infty(0,T;L^p(\OO))\,.\ee    For   $q\geq p$, set $\mu:=\frac n2\left(\frac1p-\frac1q\right)$,   the following    hold with a constant $c$ independent of $v$:
 \be\label{CT-iii}\dm \n v(t)\dm_q\leq cg_p(t-\sigma)\dm \n v(\sigma)\dm_p\,,\mbox{ for all }t>\sigma\geq0\,,\ee where  \be\label{CT-iv} g_p(t-\sigma):=\hskip-0.1cm\left\{\hskip-0.2cm\ba{lll}                \null\hskip0.2cm\frac1c\,,&t-\sigma>0\,,&\mbox{if }q=p=2\,,\\ (t-\sigma)^{-\mu},&t-\sigma\in(0,1)\,,&\null \\(t-\sigma)^{\frac12-\mu},&t-\sigma>1 \,,&\mbox{if } p\ne2\,,\,n=2\,,
 \\\log^\frac32(t-\sigma+e)\,,&t-\sigma\geq1\,,& \mbox{if }q=p=n=3\,,\\\log (t-\sigma+e)\,,&t-\sigma\geq1\,, &\hskip-0.24cm\left\{\hskip-0.2cm\ba{l}\mbox{if }q>p=n=3\,,\\\mbox{if }q\geq p=n>3\,, \ea\right.\\  (t-\sigma)^{\frac12-\frac n{2p}}\,,&t-\sigma\geq1\,,&\mbox{if }q\geq p> n\geq3   \,, \ea\right.  \ee and, for $n\geq2$, $p\in(1,n)$,
\be\label{CT-ivb}g_p(t-\sigma)\!:=\!\left\{\hskip-0.3cm\ba{lll}\null\hskip0.2cm1,&t-\sigma\!\geq1,&\mbox{if }q\geq p=2\,,\,n=2\,,\\(t-\sigma)^{-\mu}\,,&t-\sigma\geq\!1,&\mbox{if }q\in[p,n)\,,\,n>2\,,\\ (t-\sigma)^{\frac12-\frac n{2p}},&t-\sigma\geq\!1,&\mbox{if }q\!\geq\!n\,, \, p\in(1,n)\,,\,n\geq3\,,\,q\ne3\,,\\ (t-\sigma)^{\frac14-\frac3{4p}-\frac\vartheta2},&t-\sigma\geq1\,,&\mbox{if }q=n=3\,,\vartheta\in(0,\frac32\big[\frac1p-\frac13\big])\,.\ea\right.\ee
 Finally, for $q\geq p>1$, we have \be\label{TD}\dm v_t(t)\dm_q\leq c (t-\sigma)^{-\frac12-\mu} \dm \n v(\sigma)\dm_p\,, \mbox{\,for\,all }t>\!\sigma\!\geq0\,.\ee }\end{tho}
  The following result ensures that in  a suitable sense   the estimates \rf{CT-iii}   for $t \geq1$ with $g_p$ defined by \rf{CT-ivb}$_{2}$-\rf{CT-ivb}$_4$ are sharp in suitable sense and that $g_p$ in  \rf{CT-iv}$_5$-\rf{CT-iv}$_6$ for $q>p$    cannot be substituted with a $\xi(t)$ such that $t^{-1}\xi(t)\in L^1(t_0,\infty)$, $t_0\geq1$.\begin{prop}\label{CCT}{\sl For the solutions of Theorem\,\ref{CTE}, we get\begin{itemize}\item[i.]for $n>2$ and $p\in[\frac n2,n)$,  estimate \rf{CT-iv} with $g_p$ defined by \rf{CT-ivb}$_{2,3}$     is sharp,  in the sense that  
there is no function $\xi(t)$ such that 
\be\label{CT-vi}\ba{l}\null\hskip0.5cm\xi(t)t^{-1}\in L^1(t_0,\infty)\mbox{ and }\dm \n v(t)\dm_{L^q(\OO\cap S_R)} \leq \xi(t)\dm \n v_0\dm_p\left\{\hskip-0.15cm\ba{l}t^{-\mu} \,,\\ t^{\frac 12-\frac{n}{2p}}\,, \ea\right.\\\mbox{\hskip-0.72cm and for }q=n=3,\,\delta\in(0,\frac32\big[\frac1p-\frac13\big]),\\\xi(t)t^{-1+\delta}\in L^1(t_0,\infty)\mbox{ and } \dm \n v(t)\dm_{L^3(\OO\cap S_R)}\leq \xi(t)\dm \n v_0\dm_pt^{\frac14-\frac3{4p}-\frac\vartheta2}\,,\ea\ee    where $R>diam(\OO^c)$ and $\xi$ are independent of $v$.
\item[ii.] for all $q\geq p\geq n>2 $ and $q\geq p>2,n=2$\,,   on the right hand side of the estimate \rf{CT-iv} the function $g_p$  cannot be substituted by  a   function $\xi(t)$ such that $\xi(t)=o(1)$.\end{itemize}
}\end{prop} We assume  $\OO\subset\R^n$. Actually in the case of $\OO=\R^n$ the result of Theorem\,\ref{CT} is improved by estimate \rf{STC-I} of Lemma\,\ref{LSTC}. An analogous remark holds for the case of the half-space for which we refer to \cite{MSh-II} and, for $q=p=\infty$, to \cite{HK}.%\par The result of Corollary\,\ref{CCT} has the following interpretation: we cannot obtain estimates of the kind $\dm v(t)\dm_q\leq 
 \par The paper is developed on the wake of the technique adopted in \cite{MS-II}. Actually  the  arguments are essentially based on the   Green identity \rf{GIF}  (below) related to the Stokes problem. Initially, we consider the solution $U$ to the Stokes Cauchy problem  \rf{STC} (below) with   data $v_0$ extended to zero in $\R^n-\OO$. The use of the representation formula simplifies the realization of the task for the field $U$.   Then we study the Stokes initial boundary value problem \rf{STP} (below)  related to $u:=v-U$ with homogeneous initial data and boundary data $-U$. This approach leads to   estimate  the function $u$ simply by making use of the Green identity \rf{GIF} (below) written by means of the adjoint problem. The Green identity  \rf{GIF}    involves only the boundary value, that is  the trace   on $(0,T)\times\po$ of the field  $U$, and the trace on $(0,T)\times \po$ of the stress tensor of the adjoint problem that, roughly speaking, obeys the usual $L^p$-$L^q$ estimates of solutions.\par By Theorem\,\ref{CT}, $\OO$ bounded domain, we completely realize our aims. Of course, our theorem include the results of paper \cite{En} related to the heat equation with Dirichlet boundary condition.\par By Theorem\,\ref{CTE}, $\OO$ exterior domain, if  we consider $p\in(1,n)$, $n>2$, the results are, roughly speaking, in line with expectations.  In the case of $n=2$ the result is weaker. \par For $p\geq n$, the $L^p$-$L^p$ estimates of Theorem\,\ref{CTE} furnish just a result of continuous dependence.  We are neither able to prove that   $L^p$-$L^p$ estimates    holds  with $g_p\in L^\infty(0,\infty)$   nor that they do not hold for a such $g_p$. In the case of $L^p$-$L^q$ estimates, for $t\in(0,1)$ the function $g_p(t)$ in \rf{CT-iii} has the right dimensional balance, for $t>1$, the  function $g_p(t)$ in \rf{CT-iii} is growing.  However Proposition\,\ref{CCT} ensures that no decay is possible\,\footnote{\,The present statement ii. of Proposition\,\ref{CCT} improves a result previously stated by the author. It has been achieved by the author during a conversation with Prof. G.P. Galdi. Actually, by a comment on the results of the paper G.P. Galdi implicitly makes to realize the proof of the statement ii.\,.}. Different is the case of $p\in[\frac n2,n)$, where the sharpness holds as the one expressed by means of   \rf{JPAIn}$_2$ and \rf{PE}$_2$ in the case of the ordinary $L^p-L^q$ estimates.         \par The plan of the paper is the following. In sect.\,\ref{NPR} we give some notation and preliminary results concerning the trace spaces and the Gagliardo-Nirenberg inequalities,  and the solutions to the Stokes problem. In sect. \ref{US} we study the Stokes Cauchy problem assuming $v_0\in J_0(\R^n)$. In sect. \ref{uS} we study a special auxiliary Stokes IBVP. In sect. \ref{CSTF}  we furnish some implications of the results proved in sect. \ref{US} and in sect. \ref{uS}. Finally in sect. \ref{PCT} we are able to furnish the proof of   Theorem\,\ref{CT}, Theorem\,\ref{CTE} and of Proposition\,\ref{CCT}.
\section{\label{NPR}Some notation and preliminary  results} The following    spaces of completion will be considered: $J^p(\OO):=$ completion of $\mathscr C_0(\OO)$ with respect to the $L^p$-norm, $J^{1,p}(\OO):=$ completion of $\mathscr C_0(\OO)$ with respect to the $W^{1,p}(\OO)$, and, as in the introduction, $J_0^p(\OO)$:=completion of $\mathscr C_0(\OO)$ with respect to $\dm\n\cdot\dm_p$\,. We refer the reader to \cite{Gl} Theorem\,6.1 (p.68) for some properties related to the functions belonging to $J_0^p(\OO)$. \par We adopt the notations: $(u,v):=\intl\OO u \cdot v\,dx$ and $(u,v)_{\po}:=\intl\po u\cdot v\,d\omega$, respectively to mean the integral on the domain $\OO$ and the one  on the surface $\po$ related to the product of two functions $u,v$\,. The normal to the boundary is denoted by $\nu$\,.  \par By the symbol $<h>^\lambda_r$ we mean the seminorm:
$$\lambda\in(0,1)\,,\quad\big(\intl \po\intl \po
\frac{\vert h(x)-h(y)\vert^r} {\vert
x-y\vert^{n-1+\lambda r}}d\omega_x
d\omega_y\big)^{\frac1r}\,,$$
that for $\lambda=1-\frac1r$ furnish the classical one of the trace space $W^{1-\frac1r,r}(\po)$, that we consider normed by the functional
$$\dm h\dm_{1-\frac1r,r}:=\dm h\dm_{L^r(\po)}+<h>^{1-\frac1r}_r\,.$$ We recall (see e.g \cite{EG}) that, for all Lipchitz domain  $D$ such that $\ov D\cap  \po=\po$,   we get 
\be\label{tr}\ba{l}\dm h\dm_{L^r(\po)}\leq c(\dm h\dm_{L^r(D)}+\dm h\dm_{L^r(D)}^{\frac1{r'\hskip-0.1cm\null}}\hskip0.1cm\dm \n h\dm_{L^r(D)}^\frac1r)\,,\VS \dm h\dm_{1-\frac1r,r}\leq c(\dm h\dm_{L^r(D)}+\dm \n h\dm_{L^r(D)})\,,\ea\ee
with $c$ independent of $h\in W^{1,r}(D)$\,. By the symbol $W^{-\frac1r,r}(\po)$ we mean the dual space of $W^{\frac1{r'\hskip-0.1cm\null}\hskip0.1cm,r'}(\po)$. \par We denote by $T(w,\pi_w)$ the newtonian  stress tensor for a soleinodal field, and recall that $\n\cdot T(w,\pi_w)=\Delta w-\n\pi_w$\,. \par A key tool in our proof is the Green identity. If   $(\varphi,\pi_\varphi)$ is a   solutions to system  \rf{STI}$_1$, we define $\widehat\varphi(\tau,x):=\varphi(t-\tau,x)$ and $\pi_{\widehat\varphi}:=\pi_\varphi(t-\tau,x)$. It is known that $(\widehat\varphi,\pi_{\widehat\varphi})$ is a solution to the adjoint problem on $(0,t)\times\OO$, that is
\be\label{AGPR}\ba{l}\widehat\varphi_\tau+\Delta\widehat\varphi-\nabla\widehat\varphi=0,\quad\nabla\cdot\widehat \varphi=0,\mbox{ in }(0,t)\times\OO,\VS \widehat\varphi=0\mbox{ on }(0,t)\times\po,\VS \widehat\varphi=\varphi_0\mbox{ on }\{t\}\times\OO
.\ea\ee Multiplying the first equation of $(u,\pi_u)$ by $\widehat\varphi$ , and after integrating by parts on $(s,t)\times\OO$, we get the Green identity:
\be\label{GIF}\ba{l}\displ \displ(u(t),\varphi(0))+\intll st\!(\widehat\varphi,\nu\!\cdot\! T(u,\pi_u))_\po d\tau  \\\displ\hskip 4.5cm=(u(s),\widehat\varphi(t-s))+\!\!\intll st\!(u,\nu\!\cdot\! T(\widehat\varphi,\pi_{\widehat\varphi}))_\po d\tau  , \ea\ee
where the symbol $\nu$ denotes the normal on $\po$ and $s\in[0,t)$.
  \par
We recall some results that will be crucial for our aims.
\begin{lemma}\label{GN}{\sl  Let $\OO$
be a bounded  domain with the cone
property.
 Let $m\in\N$
and let $r\in[1,\infty)$ and $q\in[1,\infty]$. Let
$u\in L^q(\OO)$ and, for
$|\alpha|=m$, $D^\alpha u\in
L^r(\OO)$. Then there exists a
constant $c$ independent of $u$
such that \be\label{GN-I} \dm D^\beta
u\dm_p\leq c \dm D^\alpha
u\dm_r^a\dm u\dm_q^{1-a}+c_0 \dm u\dm_q \,,\ee
provided that for $j:=|\beta|$ the
following relation
holds:\vskip0.2cm\par\noindent
\centerline{$\frac1p=\frac
jn+a\left(\frac 1r-\frac
mn\right)+(1-a)\frac1q\,,$}\vskip0.2cm\noindent
with $a\in[\frac jm,1]$ either if
$p=1$ or if $p>1$ and $m-j-\frac
nr\notin \N\cup\{0\}$, while $a \in
[\frac jm,1)$ if $p>1$ and
$m-j-\frac
nr\in\N\cup\{0\}$. Finally, if $u\in W^{m,r}_0(\OO)$, then we can set   $c_0=0$ in \rf{GN-I}.}\end{lemma}
\begin{lemma}\label{ICM}{Let $\OO$
be an exterior domain with the cone
property.
 Let $m\in\N$
and let $q,r\in[1,\infty)$. Let
$u\in L^q(\OO)$ and, for
$|\alpha|=m$, $D^\alpha u\in
L^r(\OO)$. Then there exists a
constant $c$ independent of $u$
such that \be\label{CM} \dm D^\beta
u\dm_p\leq c\dm D^\alpha
u\dm_r^a\dm u\dm_q^{1-a}\,,\ee
provided that for $j:=|\beta|$ the
following relation
holds:\vskip0.2cm\par\noindent
\centerline{$\frac1p=\frac
jn+a\left(\frac 1r-\frac
mn\right)+(1-a)\frac1q\,,$}\vskip0.2cm\noindent
with $a\in[\frac jm,1]$ either if
$p=1$ or if $p>1$ and $m-j-\frac
nr\notin \N\cup\{0\}$, while $a \in
[\frac jm,1)$ if $p>1$ and
$m-j-\frac
nr\in\N\cup\{0\}$.}\end{lemma} The
above lemma, proved in \cite{CM},
gives an interpolation inequality
of
Gagliardo-Nirenberg's type \rf{GN-I} with $c_0=0$ in exterior domains. The difference
with respect to the usual result is
the fact that the function $u$ does
not belong to a completion space of
$C^\infty_0(\OO)$.
\begin{lemma}\label{AES}{\sl Let $D^2u\in L^q(\OO)$ and, for all bounded $\OO'\subset\OO$ such that $\po \cap\partial(\OO-\OO')=\emptyset$, assume that $u\in W^{1,q}(\OO')$ with zero trace on $\po$. Finally, assume that $\nabla \cdot u=0$ almost everywhere. Then there exists  a pressure field $\pi_u$ and a constant $c$ independent of $u$ such that
\be\label{SAE} \dm D^2u\dm_q+\dm\nabla \pi_u\dm_q+ \dm u\dm_{W^{1,q}(\OO')}\leq c(\dm P_q\Delta u\dm_q+\dm u\dm_{L^q(\OO')}).\ee   If $\OO$ is a bounded domain, then we get
\be\label{SAEI} \dm D^2u\dm_q+\dm\nabla \pi_u\dm_q+ \dm u\dm_{W^{1,q}(\OO)}\leq c\dm P_q\Delta u\dm_q ,\ee with $c$ independent of $u$ and depending on $\OO$.}\end{lemma}\begin{proof} For the proof  see for example \cite{Gl} or \cite{MSii}.
\end{proof}
Let us consider the Stokes homogeneous  problem
\be\label{SSTP}-\Delta V+\n \pi_V=0,\quad \nabla \cdot V=0\mbox{ on }\OO,\mbox{ with }V=0\mbox{ on }\po\,.\ee\begin{tho}\label{SSTPT}{\sl Let $\OO$ be an exterior domain. Let $p\geq n>2$ or $p>2$ if $n=2$. Then problem \rf{SSTP} admits a regular non trivial solution $(V,\pi_V)\in J_0^p(\OO)\times L^p(\OO)$.}\end{tho}\begin{proof} See Lemma\,5.1 in \cite{Gl}\,.\end{proof}
\begin{lemma}\label{CLM}{\sl Let $\Phi\in C(\ov \OO)\cap L^q(\OO)$ with $\intl{\po}\Phi\cdot\nu d\sigma=0$. Assume that $\n\cdot\Phi=0$ in weak sense. If the following holds  $$|(\Phi,v_0)|\leq M\dm v_0\dm_{q'}\,,\mbox{ for all }v_0\in\mathscr C_0(\OO)\,,$$ then there exists a constant $c$ independent of $\Phi$ such that \be\label{CLM-I}\dm \Phi\dm_q\leq c(M+\dm \gamma_{tr}(\Phi\cdot \nu)\dm_{-\frac1q,q})\,.\ee  }\end{lemma}\begin{proof} By virtue of the Helmholtz decomposition, for all $\psi\in  C_0(\OO)$  we get $\psi=P\psi+\n\Pi_\psi\,$ with $\dm P\psi\dm_{q'}+\dm \n\Pi_\psi\dm_{q'}\leq c\dm \psi\dm_{q'}$.  Hence we get \be\label{CLM-II}|(\Phi,\psi)|\leq |(\Phi,P\psi)|+|(\Phi,\n\Pi_\psi)|=:I_1+I_2\,.\ee
By the assumption we   deduce $I_1\leq M\dm\psi\dm_{q'}$. Instead for $I_2$, via the assumption of zero flux for $\Phi$, applying the trace theorem, we get
$$\ba{ll}I_2\hskip-0.2cm&=|(\Phi\cdot \nu,\Pi_\psi\!-\ov \Pi_\psi)_{\partial \Omega}|\leq \dm \gamma_{tr}(\Phi\cdot \nu)\dm_{-\frac1q,q}\dm \Pi_\psi\!-\ov \Pi_\psi\dm_{1-\frac{1\hskip-0.1cm\null}{q'},q'}\VSE \leq\! c\dm \gamma_{tr}(\Phi\!\cdot \!\nu)\dm_{-\frac1q,q}\dm \Pi_\psi\!-\ov \Pi_\psi\dm_{W^{1,q'}(\OO')}\!\leq\! c\dm \gamma_{tr}(\Phi\!\cdot\! \nu)\dm_{-\frac{1}{q},q}\dm \n\Pi_\psi \dm_{q'}\VSE               \leq c\dm \gamma_{tr}(\Phi\!\cdot\! \nu)\dm_{-\frac{1}{q},q}\dm \psi\dm_{q'}\,,\,\ea$$ where we applied the Poincar\é inequality after setting $\ov\Pi_\psi:=\frac 1{|\OO'|}\intl{\OO'}\pi_\psi dx$ with $\OO'\subset\OO$ and $\partial (\OO-\OO')\cap\partial \OO=\emptyset$. Estimating the right hand side of \rf{CLM-II} by means of the estimates  deduced for $I_1$ and $I_2$, since $\psi$ is arbitrary we easily arrive at \rf{CLM-I}. \end{proof}
\begin{lemma}\label{PM}{\sl Let $n\geq 3$ and $\Phi\in L^p(\OO)$, $p>\frac n{n-1}$\,. Assume that $\n\cdot\Phi=0$ in weak sense and  $$|(\Phi,v_0)|\leq M\dm v_0\dm_{q'}\,,\mbox{ for all }v_0\in\mathscr C_0(\OO)\,,$$ for some $q'>p'$, then $\Phi\in L^q(\OO)$.}\end{lemma}\begin{proof} See Lemma\,2.6 p.406 of \cite{MS-II}.\end{proof}
 Concerning the Stokes problem \rf{STI} we recall the following
\begin{lemma}\label{LI}{\sl Let $s\in(1,\infty)$. For all $\varphi_0\in  J^s(\OO)$ there exist a unique solution to problem \rf{STI} such that   
\be\label{JPAA}\ba{l}\eta>0,\; \varphi\in
C([0,T)J^s(\OO))\cap
L^\infty(\eta,T;J^{1,s}(\OO)\cap
W^{2,s}(\OO)),\\
\nabla\pi_\varphi,\,\varphi_t\in
L^\infty(\eta,T;L^s(\OO)).\ea \ee
Moreover, for $q\in[s,\infty]$ and $t>\sigma\geq0$, set $\mu:=\frac n2\left(\frac1s-\frac1q\right)$, we get
\be\label{JPA}\ba{ll}\dm \varphi(t )\dm
_q\leq c\dm \varphi(\sigma)\dm _s(t\!-\!\sigma)^{-\mu},\hskip-0.17cm&
 \mbox{for all }t-\sigma>0;\vspace{4pt}\\\dm
\nabla\! \varphi(t )\dm _q\!\leq \!c\dm \varphi(\sigma)\dm
_s(t\!-\!\sigma)^{-\mu_1},\hskip-0.17cm
&\hskip-0.1cm
\mu_1\hskip-0.1cm:=\!\left\{\hskip-0.2cm\ba{ll}\frac12\!+\!\mu\hskip-0.55cm
&\mbox{if }t\!-\!\sigma\in(0,1],\\\frac12\!+\!\mu
&\mbox{if }t\!-\!\sigma\!>\!1,q\leq\! n,
\\\frac n{2s}&\mbox{if }t\!-\!\sigma\!>\!1\,,q\!\geq n;\ea\right.
\vspace{4pt}\\\dm \varphi_t(t )\dm _q\leq
c\dm \varphi(\sigma)\dm _s(t\!-\!\sigma)^{-\mu_2},\hskip-0.17cm
& \hskip-0.1cm\mu_2\hskip-0.1cm:=1+\mu,\mbox{ for all }t-\sigma>0;\vspace{4pt}\\\dm D^{ 2}\!\varphi(t )\!\dm_q\leq\! c \dm \varphi(\sigma)\dm_s(t\!-\!\sigma)^{-\mu_3}, \hskip-0,17cm&\hskip-0.1cm\mu_3\hskip-0.1cm:=\!\left\{\hskip-0.2cm\ba{ll}1\!+\!\mu\hskip-0.25cm&\mbox{if }t\!-\!\sigma\in(0,1],\\1\!+\!\mu& \mbox{if }t\!-\!\sigma\!>\!1,q<\!\frac n2,
\\\frac{n}{2s}&\mbox{if }t\!-\!\sigma>\!1, q\!\geq\!\frac n2;\ea\right.
\ea\ee
where the constant $c$ is
independent of $\varphi_0$ and the
exponent $\mu_1$ is sharp for $s\geq\frac n2\,,\,n\geq3$ in the sense that there is no function $\xi(t)$ such that 
\be\label{OPT-I}t^{-1}\xi(t)\in L^1(t_0,\infty)\mbox{ and }\dm \n \varphi(t)\dm_{L^q(\OO\cap S_R)}\leq \xi(t) t^{-\mu} \dm \varphi_0\dm_s \,,\ee where $R>diam(\OO^c)$ and $\xi$ are independent $\varphi$. Finally, for all $s\in(1,\infty)$ and $\varphi_0\in J^s(\OO)$ the following limit property holds:
\be\label{FPL}\lim_{t\to\infty}\dm \varphi(t)\dm_s=0\,.\ee   }\end{lemma}\begin{proof} With exception of \rf{JPA}$_1$ in the case of  $n=2$ and $q=\infty$, for which we refer to \cite{DS-I,DS-II}, the   claims of the lemma are essentially the ones proved in 
  \cite{MS-II}. Estimate \rf{JPA}$_4$ is contained in \cite{MS-II} but it is not stated in no theorem. However, after remarking that $P\Delta \varphi=\varphi_t$, for the task it is enough to apply  estimate \rf{SAE} and suitably estimates \rf{JPA}$_{1,2,3}$\,. As well the optimality expressed by \rf{OPT-I} is an improvement of the ones given in \cite{MS-II} (see also \cite{DS-I,DS-II,GT}). We furnish the proof of the optimality stated by means of \rf{OPT-I} in Lemma\,\ref{AX} below.  \end{proof}\begin{coro}\label{BDC}{\sl In the same hypotheses of Lemma\,\ref{LI} and furthermore assuming $\OO$ bounded domain, then for $t>\sigma\geq0$ the following holds:
\be\label{BDC-I}(t-\sigma)\dm \varphi_t(t)\dm_q+(t-\sigma)^\frac12\dm \n \varphi(t)\dm_q+\dm \varphi(t)\dm_q\leq c\dm \varphi(\sigma)\dm_se^{-c_1(t-\sigma)}(t-\sigma)^{-\mu}, \ee where $c_1$ is a constant depending on the size of $\OO$ and constants $c,c_1$ are independent of $\varphi_0$ and of $t,\sigma$. }\end{coro}\begin{coro}\label{CSP}{\sl Let $\varphi_0\in \mathscr C_0(\OO)$. Then, for all $\eta>0$, the solution of Lemma\,\ref{LI} is such that
  \be\label{INT}\ba{l}\varphi\in\mbox{${\underset{s>1}\cap}$}\Big[C([0,T);J^s(\OO))\cap L^\infty(\eta,T;J^{1,s}(\OO)\cap W^{2,s}(\OO))\Big]\,,\VS
\n\pi_\varphi,\,\varphi_t\in \mbox{${\underset{s>1}\cap}$}L^\infty(\eta,T;L^s(\OO))\,.\ea\ee}\end{coro}
\begin{proof} For the proof of the Corollary see e.g. \cite{MS-II}\,.\end{proof}
\begin{lemma}\label{TSB}{\sl Let $(\varphi,\pi_\varphi)$ be the solution of Lemma\,\ref{LI}. For $r\geq s$ the following estimates hold:  \be\label{TSB-I}\ba{l}t-\sigma\in(0,1)\,,\quad\dm \n \varphi(t )\dm_{L^r(\po)}\leq c \dm \varphi(\sigma)\dm_s(t-\sigma)^{-\mu_4 }\,, \VS t-\sigma>1\,,\quad \dm\n\varphi(t ))\dm_{L^r(\po)}\leq c\dm \varphi(\sigma)\dm_s(t-\sigma)^{-\mu_5}, \ea\ee  where $c$ is a constant  independent of $\varphi_0$ and $t-\sigma$\,, furthermore  we have set $\mu_4:= \frac12+\mu+\frac1{2r}$ and $\mu_5:=\left\{\hskip-0.2cm\ba{ll}\frac12+\mu &\mbox{if }r\in[s, n]\,,\vspace{2pt}\\   \frac n{2s}&\mbox{if }r\geq n\,, \ea \right.$.}\end{lemma}\begin{proof} Estimate \rf{TSB-I} is an immediate consequence of the trace inequality \rf{tr} and of estimates \rf{JPA}$_{2,4}$.\end{proof}\begin{lemma}\label{PSl}{\sl Let $(\varphi,\pi_\varphi)$ be the solution of Lemma\,\ref{LI}. Then the pressure field $\pi_\varphi$ enjoys the estimates:\be \label{prss}\begin{array}{ll}
\lambda\in(0,1),\;\vert
\pi_\varphi\vert_{L^r(\OO\cap B_R)}\hskip-0.3cm&\leq c<\nabla
\varphi>^\lambda_{r}\,,\;
\;\vspace{6pt}\\\null\hskip2.8cm\vert \nabla \pi_\varphi\vert_r\hskip-0.3cm&\leq
c<\nabla\varphi>^{1-\frac1r}_r\,,\end{array}\ee with $c$ independent of $\varphi$. In particular, if $r>n\geq2$ we get \be\label{PPE}\ba{ll}\null&|\pi_\varphi(x)|\leq c(\dm \n\varphi\dm_r+\dm\n\n\varphi\dm_r)|x|^{2-n}\,,\quad |x|>R\,,\vspace{3pt}\\\mbox{for } n=2,&\pi_\varphi-\pi_\infty =o(\dm \n\varphi\dm_r+\dm\n\n\varphi\dm_r) \, .\ea\ee }\end{lemma}\begin{proof} The estimates \rf{prss} are consequence of the results due to Solonnikov in \cite{Sl} related to the Neumann problem:
$$\ba{c}\Delta\pi_\varphi=0\mbox{ in }\OO\,,\quad \frac d{d\nu}\pi_\varphi=\nu\cdot \Delta\varphi\,\mbox{ on }\po\,,\\n=2,\;\pi_\varphi\to \pi_\infty\,,\;n\geq3,\;\pi_\varphi\to0\,,\;\;|x|\to\infty.\ea$$ 
  \end{proof}The following lemma furnishes the behavior in $t$ related to a trace-norm of the pressure field $\pi_\varphi$. The behavior   depends on the neighborhood of $t=0$ and of $t=\infty$. Of course, our task is to deduce   behavior that turns to be the best for our aims.  
 \begin{lemma}\label{PT}{\sl Let $(\varphi,\pi_\varphi)$ be the solution of Lemma\,\ref{LI}. Then, for $r\geq s$, set $\mu:=\frac n2\left(\frac 1s-\frac1r\right)$, we get
\be\label{PT-I}\ba{l}t-\sigma\in(0,1),\quad\dm \pi_\varphi(t)\dm_{L^r(\po)}\leq c\dm\varphi(\sigma)\dm_s(t-\sigma)^{-\rho_0-\mu}\,,\VS t-\sigma>1\,,\quad\dm \pi_\varphi(t)\dm_{L^r(\po)}\leq c\dm\varphi(\sigma)\dm_s(t-\sigma)^{-\rho_1-\mu}\,, \ea\ee where $c$ is a constant  independent of $\varphi $ and $t-\sigma$, and   $$\ba{l}\rho_0:=\frac12+ \frac{1+r(\lambda-1)}{2r^2(n-2+r)}+  \frac{r(1-\lambda)-1}{r(n-2+r)}+\frac{n-1+\lambda r}{2(n-2+r)}<1\,,\vspace{4pt}\\\rho_1:= \left\{\!\!\ba{ll}\frac12+\frac1{2d}&\mbox{if }r\in(1,\frac n2]\,,\\ \frac12+\frac1{2d}\!\left(\frac nr-1\right)&\mbox{if }r\in[\frac n2,n]\,,\\ \frac n{2r}&\mbox{if }r>n\,,\ea\right.\ea$$ where    $\lambda\in(0,1-\frac1r)$ and $d:=\frac{n-2+r}{n-1+\lambda r}$\,.} \end{lemma}\begin{proof}  
  \vskip
0.05 true cm\par\noindent We prove estimates \rf{PT-I} for $\sigma=0$ and $s=r$. Subsequently, one deduces estimates \rf{PT-I} in a complete form by means of the semigroup properties of $\varphi$.  Assuming in \rf{prss} $\lambda<1-\frac 1r,$
applying
H\"older's inequality with exponents $(d,\frac
d{d-1}),d=\frac{n-2+r}{n-1+\lambda r},$ (we stress that $1<d<r$) we get\be\label{trs}\begin{array}{ll}\displ\big(
\! \! <\nabla \varphi>^\lambda_r\! \!
\big)^r\hskip-0.3cm&= \displ\!  \intl \po \intl
\po\! \vert\nabla
\varphi(x)-\!\nabla\varphi(y)\vert
^{r(1-\frac1d)}
\frac{\vert\nabla\varphi(x)-\! \nabla\varphi(y)
\vert^{\frac rd\hskip-0.2cm\null}}{\vert
x{-}y\vert^{n-1+\lambda
r}}d\sigma_xd\sigma_y \vspace{7pt}\VSE \leq
c\vert\nabla\varphi\vert_{L^r(\po)}^{r(1-
\frac1d)} (<\nabla
\varphi>^{1-\frac1r}_r)^{\frac rd}\end{array}
\ee Employing  the trace inequality   \rf{tr}$_1$ and estimates
\rf{prss}-\rf{trs},
$(1-\frac1d=\frac1{d'}$ and $1-\frac
1r=\frac1{r'})$ we deduce
\be\label{traa}\begin{array}{l} \dm
\pi_\varphi\dm_{L^r(\po)}\leq c\dm  
\pi_\varphi\dm _{L^r(\OO\cap S_R)}+\dm 
\pi_\varphi\dm _{L^r(\OO\cap S_R)}^{\frac
1{r'\hskip-0.1cm\null}\hskip0.1cm}\dm  \nabla \pi_\varphi\dm _{L^r(\OO\cap
S_R)}^{\frac1r}\vspace{6pt}\\\hskip 3 true
cm \leq c\big(\!<\!
\nabla\varphi>^\lambda_r+(<\! \nabla\varphi >^
\lambda_r)^{\frac1{r'\hskip-0.1cm\null}\hskip0.1cm}(<\!
\nabla\varphi>_r^{\frac1{r'\hskip-0.1cm\null}\hskip0.1cm})
^{\frac1r}\big)\vspace{7pt}\\
\leq\!
c\Big[\dm \nabla\varphi \dm _{_{L^r(\po)}}^
{\frac1{d' \hskip-0.1cm\null}\hskip0.1cm} (<\! \nabla \varphi\!
>^{\frac1{r'\hskip-0.1cm\null}\hskip0.1cm}_r)^{\null^\frac1
d}\!\! +\!\dm \nabla\varphi\dm _{_{L^r(\po)}}^
{\frac1{d'\hskip-0.1cm\null}\hskip0.1cm\frac1{r'\hskip-0.1cm\null}\hskip0.1cm} (<\! \nabla \varphi\!
>^{\frac1{r'\hskip-0.1cm\null}\hskip0.1cm}_r)^{\null^{\frac1
d\frac1{r'\hskip-0.1cm\null}\hskip0.1cm}}\!(<\! \nabla \varphi\!
>^{\frac1{r' \hskip-0.1cm\null}}_r)^{\null^\frac1
r}\Big]\vspace{6pt}\\\hskip 4 true cm
=I_1(r,t)+I_2(r,t).
\end{array}\ee
Employing again  the trace inequality 
\rf{tr}$_2$, we get
\be\label{trbb}\begin{array}{ll}\null  I_1(r,t)& \hskip-0.3cm\leq
c\big(\dm  \nabla\varphi   \dm _{\null_{L^r(\OO\cap
S_R)}} +
\dm \nabla\varphi \dm _{\null_{L^r(\OO\cap
S_R)}}^{\frac 1{r'\hskip-0.1cm\null}\hskip0.1cm}\dm  D^2
\varphi \dm _{\null_{L^r(\OO)}}^{\frac1r}\big)^{\frac
1{d'\hskip-0.1cm\null}\hskip0.1cm}\dm 
D^2\varphi \dm _{\null_{L^r(\OO)}}^{\frac1d} \big)\VS & \hskip-0.3cm\leq
c\big(\dm \nabla\varphi 
\dm _{\null_{L^r(\OO\cap
S_R)}}^{\frac1{d'\hskip-0.1cm\null}\hskip0.1cm}\dm 
D^2\varphi \dm _{\null_{L^r(\OO)}}^\frac1d+
\dm \nabla\varphi \dm _{\null_{L^r(\OO\cap
S_R)}}^{\frac1{r'\hskip-0.1cm\null}\hskip0.1cm\frac1{d'\hskip-0.1cm\null}\hskip0.1cm}\dm  D^2
\varphi \dm _{\null_{L^r(\OO)}}^{\nu}\big)\vspace{2pt}\VS 
 I_2(r,t)&\hskip-0.3cm\leq
c\big(\dm \nabla\varphi \dm _{\null_{L^r(\OO\cap
S_{\!R})}}\! \! +
\dm \nabla\varphi \dm _{\null_{L^r(\OO\cap
S_{\!R})}}^{\frac 1{r'\hskip-0.1cm\null}\hskip0.1cm}\dm  D^2
\varphi \dm _{\null_{L^r(\OO)}}^{\frac1r}\big)^{\frac
1{d'\hskip-0.1cm\null}\hskip0.1cm\frac1{r'\hskip-0.1cm\null}\hskip0.1cm}\dm 
D^2\varphi \dm _{\null_{L^r(\OO)}}^{\nu
} \VS &\hskip-0.3cm\leq c\big(\dm \nabla\varphi 
\dm _{\null_{L^r(\OO\cap
S_{\!R})}}^{\frac1{d'\hskip-0.1cm\null}\hskip0.1cm\frac1{r'\hskip-0.1cm\null}\hskip0.1cm}\dm 
D^2\varphi \dm _{\null_{L^r(\OO)}}^{\nu}\! \!
+\!\dm \nabla\varphi \dm _{\null_{L^r(\OO\cap
S_{\!R})}}^{ (\frac1{r'\hskip-0.1cm\null}\hskip0.1cm\! )^2\frac1{d'}}\dm 
D^2
\varphi \dm _{\null_{L^r(\OO)}}^{\nu_1+\nu}\big) \VS\mbox{where}&\mbox{\hskip-0.28cm we have set } 
   \nu :=\mbox{$\frac1d\frac1{r'\hskip-0.1cm\null}\hskip0.1cm+
\frac1r$ and  $\nu_1:=\frac1r\frac
1{d'\hskip-0.1cm\null}\hskip0.1cm\frac1{r'\hskip-0.1cm\null}\hskip0.1cm$}\,.\ea \ee
\vskip 0.05 true cm \par \noindent
For the right hand side of \rf{trbb} we look for an estimate in $t$ and $\dm \varphi _0\dm _{r}$.  
We firstly evaluate
$I_1(r,t)$ e $I_2(r,t)$ for
 $t\in(0,1).$ We estimates the terms on the right hand side of \rf{trbb} by inequalities \rf{JPA}$_{2,4}$. Since we evaluate for $t\in (0,1)$, we can limit ourselves to consider the terms on the right hand side of 
\rf{JPA}$_{2,4}$ which have max exponent. This max exponent is leaded by the last term of $I_2(r,t)$. Hence we have  \be\label{opi1}I_1(r,t)+I_2(r,t)\leq c\dm \varphi_0 \dm_rt^ {-\rho_0},\;t\in(0,1),\ee  where we have set $\rho_0:=\frac 12\frac1{r'}\frac1{r'}\frac1{d'}+\frac1r\frac1{r'}\frac1{d'}+\frac1d\frac1{r'}+\frac1r$\,. We compute $\rho_0$. Recalling that $r',\,d'$ are the coniugate exponents of $r,\,d$, we get\,\footnote{\, Actually it holds$$\ba{ll}\rho_0\hskip-0.2cm&=\left(\frac1{2r'}+\frac1{r}\right)\frac1{d'r'}+\frac1{dr'}+\frac1r=\left(\frac12+\frac1{2r}\right)\frac1{d'r'}+\frac1{dr'}+\frac1r \vspace{3pt}\\&=\frac12\left(1+\frac1{r}\right)\frac1{r'}-\frac1{2d}\left(1+\frac1r\right)\frac1{r'}+\frac1{dr'}+\frac1r\ea$$ that   leads  \rf{com} substituting again $\frac1{r'}$ with $1-\frac1r$\,.}
\be\label{com}\mbox{$\rho_0=\frac12 +\frac12\frac1{r^2}\left(\frac1d-1\right)+\frac1r\left(1-\frac1d\right)+\frac1{2d}$}\,.\ee By the definition of $d$, we have that \rf{com} is equivalent to  $$\mbox{$\rho_0=\frac12+ \frac{1+r(\lambda-1)}{2r^2(n-2+r)}+  \frac{r(1-\lambda)-1}{r(n-2+r)}+\frac{n-1+\lambda r}{2(n-2+r)}\,.$}$$ We are interested to verify that under our assumption on $\lambda$ and for $r>1$ we get $\rho_0<1$\,, that is
\be\label{comi}\mbox{$\frac12+ \frac{1+r(\lambda-1)}{2r^2(n-2+r)}+  \frac{r(1-\lambda)-1}{r(n-2+r)}+\frac{n-1+\lambda r}{2(n-2+r)}<1$}\,.\ee
   For $\vep>0$, we set $r:=1+\vep$. Hence \rf{comi} becomes equivalent to\,\footnote{\, Estimate \rf{comi} is equivalent to \mbox{$  \frac{1+r(\lambda-1)}{2r^2(n-2+r)}+  \frac{r(1-\lambda)-1}{r(n-2+r)}+\frac{n-1+\lambda r}{2(n-2+r)}<\frac12$}, which is equivalent to  $   1+r(\lambda-1) +  2r^2  (1-\lambda)-2r   <  -r^2+(1-\lambda)r^3$. Introducing $r:=1+\vep$ we obtain the first of \rf{comii}\,.  }
\be\label{comii}\lambda < (1-\lambda)\vep \;\Leftrightarrow\;\lambda<(1-\lambda)(r-1) \;\Leftrightarrow\;\lambda<1-\frac1r\,.\ee Since it is $\lambda<1-\frac1r$\,, we have verified \rf{comi}.
Now we look for the estimate of $I_1(r,t)$ and $I_2(r,t)$ for $t>1$. Since exponent   in \rf{JPA}$_{2,4}$ depends on $r$ we distinguish the cases of $r\in (1,\frac n2]$ from the ones of $r\in[\frac n2,n]$ and $r>n$. Suppose $r\in(1,\frac n2]$.  Since $t>1$  we look for exponents minimum. Hence, evaluating the right hand side of \rf{trbb} via the estimates \rf{JPA}$_{2,4}$, we get \be\label{opi2}r\mbox{$\in(1,\frac n2]$}\,,\quad 
I_1(r,t)+I_2(r,t) \leq c\dm \varphi _0\dm_{r}t^{-\rho_1},t>1\ee where we have set $\rho_1:=\frac12+\frac1{2d}$\,.  In the case of $r\in[\frac n2,n]$, for the right hand side of \rf{trbb}, we obtain the exponents:
$$I_1(r,t)\leq c\dm \varphi _0\dm_r\big(t^{-\rho_{11}}+t^{-\rho_{12}}\big)\,,\;t>1,$$
with $\rho_{11}:=\frac12+\frac1{2d}\left(\frac nr-1\right)$ and $\rho_{12}:=
\rho_{11}+\frac1{2r}\left(\frac nr-1\right)\left(1-\frac1d\right)$, as well $$I_2(r,t)\leq c\dm \varphi _0\dm_r\big(t^{-\rho_{21}}+t^{-\rho_{22}}\big)\,,\;t>1,$$
with $\rho_{21}:=\rho_{12}$ and $\rho_{22}:=\rho_{11}+\frac1{r}\left(\frac nr-1\right)\left(1-\frac1d\right)\left(1-\frac1{2d}\right)$. Hence choosing the minimum exponent we get \be\label{opi3}r\in\mbox{$[\frac n2,n]$},\quad I_1(r,t)+I_2(r,t)\leq c\dm \varphi _0\dm_rt^{-\rho_2}\,,\; t>1,\ee where we have set $\rho_2:=\rho_{11}$. Now we consider the case of $r>n$. Summing the  exponents of terms on the right hand side of \rf{trbb} we get   $\frac n{2r}$  as minimum exponent. Hence we get
\be\label{opi4}r>n\,,\quad I_1(r,t)+I_2(r,t)\leq c\dm \varphi  _0\dm_rt^{-\frac n{2r}}\,,\;t>1\,.\ee Finally, via estimates \rf{traa}-\rf{opi1} and \rf{traa} with \rf{opi2}-\rf{opi4}, we get\be\label{tr11}  
\begin{array}{l} \dm \pi_\varphi(t)\vert_{ {L^{  r}(\po)}}  \leq   c\dm \varphi _0\dm_{r}t^{-\rho_0} ,\;\forall
t\! \in\! (0,1]\mbox{ and for all }r>1,  \VS \dm
\pi_\varphi(t)\dm_{\null_{L^r(\po)}}  \leq 
c\dm\varphi _0\dm_rt^{-\rho_1}  ,\;\forall t 
\geq1\,.
\end{array}\ee
\end{proof}\begin{rem}\label{RSt}{\rm We point out that $\rho_0>\rho_1$. Moreover comparing  estimates \rf{TSB-I} with the estimates \rf{PT-I},  related to $\n \varphi$ and to $\pi_\varphi$ respectively, we note that,  in a neighborhood of $t=0$, $\rho_0+\mu>\frac12+\mu$ for all $r\geq s$,   as well, in neighborhood of infinity, if $r<n $ then $\rho_1+\mu>\frac12+\mu$, and if $r\geq n$ then  $\rho_1+\mu=\frac n{2s}$.}\end{rem}
The following result holds:\begin{coro}\label{StTs}{\sl 
Let $(\varphi,\pi_\varphi)$ be the solution furnished in Lemma\,\ref{LI}, then for $r\geq s$, set $\mu:=\frac n2\left(\frac1s-\frac1r\right)$\,, we get
\be\label{StTs-I} \dm T(\varphi,\pi_\varphi)(t)\dm_{L^r(\po)}\leq c\dm \varphi(\sigma)\dm_s \!\left\{\!\!\!\!\ba{ll}(t\!-\!\sigma)^{-\rho_0-\mu}&\mbox{if }t\!-\!\sigma\in(0,1)\,,\\(t\!-\!\sigma)^{-\frac12-\mu}&\mbox{if }t\!-\!\sigma>1\mbox{ and }r\!\leq\! n\,,\\ (t\!-\!\sigma)^{-\frac n{2s}}&\mbox{if }t\!-\!\sigma>1 \mbox{ and }r\!>\!n.\ea\right.\ee Finally, we also have
\be\label{StTs-ID} \dm T(\varphi_t,\pi_{\varphi_t})(t)\dm_{L^r(\po)}\!\leq\! c\dm \varphi(\sigma)\dm_s \!\left\{\!\!\!\!\ba{ll}(t\!-\!\sigma)^{-\rho_0-\mu-1}\hskip-0.15cm&\mbox{if }t\!-\!\sigma\in(0,1)\,,\\(t\!-\!\sigma)^{-\frac32-\mu}&\mbox{if }t\!-\!\sigma\!>\!1\mbox{\,and }r\!\leq\! n,\\ (t\!-\!\sigma)^{-\frac n{2s}-1}&\mbox{if }t\!-\!\sigma\!>\!1 \mbox{\,and }r\!>\!n.\ea\right.\ee}\end{coro}\begin{proof} The proof is an immediate consequence of Lemma\,\ref{PSl}, Lemma\,\ref{PT} and Remark\,\ref{RSt}.\end{proof} 
\begin{lemma}\label{AX}{\sl Let $(\varphi,\pi_\varphi)$ be a solution to problem \rf{STI} given by Corollary\,\ref{CSP}. Then, for $s\geq\frac n2$, there is no function $\xi(t)$ such that 
\be\label{OPT-II}t^{-1}\xi(t)\in L^1(t_0,\infty)\mbox{ and }\dm \n \varphi(t)\dm_{L^q(\OO\cap S_R)}\leq \xi(t) t^{-\mu_1} \dm \varphi_0\dm_s \,,\;t>0\,,\ee where $R>diam(\OO^c)$ and $\xi$ are independent of $\varphi$ and $\mu_1$ is given in \rf{JPA}$_2$.   }\end{lemma}\begin{proof} Let $\varphi_0\in \mathscr C_0(\OO)$ and let $(\varphi,\pi_\varphi)$ be the solution ensured by Corollary\,\ref{CSP}. 
Employing \rf{tr} and subsequently \rf{SAE},   for all $R>diam(\OO^c)$, we have
\be\label{OPT-IV}\ba{ll}\dm \n \varphi(t)\dm_{L^{ q}(\po)}\hskip-0.2cm& \leq c(\dm \n\varphi(t)\dm_{L^q(\OO\cap S_R)}+\dm \n \varphi(t)\dm_{L^q(\OO\cap S_R)}^{\frac1{q'\hskip-0.1cm \null}}\dm D^2\varphi(t)\dm_{L^q(\OO\cap S_R)}^{\frac1{ q}})\\& \leq c(\dm \n\varphi(t)\dm_{L^q(\OO\cap S_R)}+ \dm P\Delta\varphi(t)\dm_{ q} + \dm \varphi(t)\dm_{L^q(\OO\cap S_R)} )\\&\leq  c(\dm \n\varphi(t)\dm_{L^q(\OO\cap S_R)}+ \dm \varphi_t(t)\dm_{ q} )\,. \ea \ee
We point out that\begin{itemize}\item  estimating the penultime row, for the second term, we toke into account   that      the equation \rf{STI}$_1$ furnishes $P\Delta \varphi=\varphi_t$,\item  since in \rf{SAE} $\OO'$ is bounded with $\po\cap\po'=\po$ we can choose $\OO'\equiv\OO\cap S_R$,         and,      for the third term, we employed the Poincar\'e inequality.\end{itemize}  Estimatye \rf{traa} ensures
$$\dm \pi_\varphi(t)\dm_q\leq I_1(t)+I_2(t)\,,\mbox{ for all }t>0\,.$$
 Recalling estimate \rf{trbb}, computing the exponents  and by making use of Young inequality, we get 
$$ I_1(q,t)+I_2(q,t) \leq c(\dm \n \varphi(t)\dm_{L^q(\OO\cap S_R)} +\dm D^2\varphi(t)\dm_q)\,. $$
By the same arguments employed in the previous computation for the $D^2\varphi$   we   obtain
\be\label{SII} \dm\pi_\varphi(t)\dm_q\leq  I_1(q,t)+\!I_2(q,t)   \leq c (  \dm \n\varphi(t) \dm_{L^q(\OO\cap S_R)}+ \dm \varphi_t(t) \dm_{ q}) 
,\mbox{ for all }t>0. \ee    
Now we are in a position to prove the lemma.  We   adapt the idea already employed in \cite{MS-II}. Let consider the exterior problem\be\label{STS} \Delta \Phi=\n P\,,\quad\n\cdot\Phi=0 \mbox{\; in \;}\OO,\;\Phi=a\mbox{ on }\po\,,\;\Phi\to0\mbox{ for }|x|\to\infty\,.\ee It is well known that assuming $a\in C^2(\po)$ there exists a solution such that $\Phi=O(|x|^{-n+2})$ at infinity. Hence $\Phi\in L^q(\OO)$ with arbitrary $q>\frac n{n-2}$. Our task is to prove that if   
\rf{OPT-II} is true, then the following holds: \be\label{OP-I}
|(\Phi,\psi_0)|\leq c\dm \psi_0\dm_{\frac n2}\,,\quad \psi_0\in \mathscr C_0(\OO)\,.\ee By virtue of Lemma\,\ref{PM}
this last implies that $\Phi\in L^\frac n{n-2}(\OO)$ for arbitrary boundary data $a\in C^2(\OO)$, which is impossible. Hence \rf{OPT-II} can not be true. Now our task is to prove \rf{OP-I} via the assumption \rf{OPT-II}. Assuming that   estimate    holds  for the solution $(\varphi,\pi_\varphi)$ to problem \rf{STI} with initial data $\varphi_0\in \mathscr C_0(\OO)$,   multiplying equation \rf{STS}$_1$ by $\varphi$ and integrating on $(0,T)\times \OO$, we obtain:
$$\ba{ll}(\Phi,\varphi_0)\hskip-0.2cm&\displ=(\Phi,\varphi(t))-\intll0{t_0}\Big[(a,\nu\cdot \n\varphi)_{\po}-(a\cdot\nu, \pi_\varphi)_\po\Big] d\tau\\&\hskip3.8cm+\displ\intll{t_0}t\Big[(a,\nu\cdot \n\varphi)_{\po}-(a\cdot\nu, \pi_\varphi)_\po\Big] d\tau\,.\ea $$  
Applying H\"older's inequality, for $\ov q>\frac n{n-2}$\,, we get \be\label{IA-O}\ba{ll}|(\Phi,\varphi_0)|\hskip-0.2cm&\displ\leq\dm \Phi\dm_{\ov q}\dm \varphi(t)\dm_{\ov q'}+c \dm a\dm_\infty\!\!\intll0{t_0}\!\Big[\dm \n\varphi\dm_{L^\frac n2(\po)}+\dm \pi_\varphi\dm_{L^\frac n2(\po)}\Big] d\tau\\&\hskip3.3cm+c\dm a\dm_\infty\displ\!\!\intll{t_0}t\!\Big[\dm \n\varphi\dm_{L^q(\po)}+\dm\pi_\varphi\dm_{L^q(\po)}\Big] d\tau,\ea \ee  here, assuming $\varphi_0\in \mathscr C_0(\OO)$, we have tacitly considered that \rf{JPAA} holds.
Applying to the right hand side of \rf{IA-O}    for $(\n \varphi,\pi_\varphi)$ estimates \rf{TSB-I}$_1$ and \rf{PT-I}$_1$  for $\tau\in(0,t_0)$, and estimates \rf{OPT-IV}-\rf{SII} for $\tau>t_0$,  by virtue of assumption \rf{OPT-II}, we get
$$ |(\Phi,\psi_0)| \displ\leq \dm \Phi\dm_{\ov q}\dm \psi(t)\dm_{\ov q'} +c \dm a\dm_\infty\dm \psi_0\dm_{\frac n2} \,,  $$
which,   letting $t\to\infty$ and  employing \rf{FPL}\,, implies \rf{OP-I}.

\end{proof}
\section{\label{US}The      Stokes Cauchy problem} Let us consider the Stokes Cauchy problem:
\be\label{STC}\ba{l}U_t-\Delta U+\n\pi_U=0\,,\quad\nabla\cdot U=0\,,\mbox{ in }(0,T)\times\R^n\,,\VS U=U_0\mbox{ on }\{0\}\times\R^n\,.\ea\ee
\begin{lemma}\label{LSTC}{\sl For all $U_0\in   J^s_0(\R^n)$ there exits, up to a function of $t$ for the pressure field, a unique solution $U$ to problem \rf{STC} such that $\n U\in C([0,T);J_0^p(\R^n))$ and,  for all $r\geq s\geq1$ and $t>\sigma\geq0$, set $\mu:=\frac n2\left(\frac1s-\frac1r\right)$, we get  \be\label{STC-I}\ba{l}\dm\n U(t )\dm_r\leq c\dm \n U(\sigma)\dm_s(t\!-\!\sigma )^{-\mu}\,,\mbox{ for all }t\!-\!\sigma>0,\VS\dm D^{ 2}\hskip-0.03cmU(t )\dm_r\!+\! \dm U_t(t )\dm_r\!\leq \!c\dm \n U(\sigma)\dm_s(t\!-\!\sigma)^{-\frac12-\mu}\!,\mbox{\,for\,all\,}t\!-\!\sigma\!>\!0, \ea\ee where the constant $c$ is independent of $U_0$\,. }\end{lemma}\begin{proof} By means of the representation of the solution by heat  kernel, after integrating by parts, and via the  Young theorem we get \rf{STC-I}$_{1,2}$.\end{proof}In the following corollary we make the special assumption of $U_0\in \mathscr C_0(\OO)\subset\mathscr C_0(\R^n)$ (we mean that $U_0$ has a trivial extension on $\R^n$), and we study the behavior of the solutions corresponding to these special initial data in neighborhood of $t=0$ and of $t=\infty$. Of course, the special data influence the quoted behavior. They are special in such a way that they are useful for our subsequent  tasks.     \begin{coro}\label{CCP}{\sl Let $U_0\in \mathscr C_0(\OO)$. Then the solution of Lemma\,\ref{LSTC} is such that \be\label{CCP-I}\ba{l}\dm U(t)\dm_\infty\leq c\dm\n U_0\dm_pt^{\frac12-\frac n{2p}},\mbox{if }p\in(1,n)\,,\vspace{2pt}\\\dm U(t)\dm_{L^\infty(\OO\cap B_R)}\leq c\dm \n v_0\dm_p\zeta_d(t) \ea\ee\be\label{CCPI-I}\ba{l}   \dm U(t)\dm_{L^\infty(\po)}+\dm U(t)\cdot\nu\dm_{-\frac1q,q}\leq c\dm \n U_0\dm_p \zeta_b(t)\,,   \vspace{2pt} \\ \dm U(t)\dm_{L^p(\po)}\leq  c \dm \n U_0\dm_pt^\frac12\,,\vspace{2pt}\\ \dm U_t(t)\cdot\nu\dm_{-\frac1q,q}\leq c \dm \n U_0\dm_pt^{-\frac12-\mu}\,, \ea\ee with $\zeta_d:= \left\{\hskip-0.2cm\ba{ll}1+\intll1t\!\frac1\tau d\tau&\mbox{if }p=n\,,\\1+t^{\frac12-\frac n{2p}}&\mbox{if }p\ne n\,,\ea\right.$ and $\zeta_b(t):=\left\{\hskip-0.2cm\ba{ll}t^{\frac12-\frac n{2p}}&\mbox{if }p\ne n\,,\\ \log(t+e)&\mbox{if }p=n\,,  \ea\right.$ where $c$ is a constant independent of $U_0$ and all the estimates hold uniformly in $t>0$. There exists a constant $c$ such that for all $U_0\in \mathscr C_0(\R^n)$ \be\label{DSDT}\dm U_t(t)\dm_p+\dm D^2U(t)\dm_p\leq c\dm D^2U_0\dm_p\,,\mbox{ for all }t>0\,.\ee }\end{coro}\begin{proof} In the case of $p\in(1,n)$ estimate \rf{CCP-I}$_1$ is an immediate consequence of \rf{STC-I}$_2$ and of the fact that \be\label{LI-I}\dm U(t)\dm_\infty\leq c\intll t\infty\dm U_\tau(\tau)\dm_\infty d\tau\,\;\mbox{ for all }t>0.\ee Analogously for $p>n$, employing again \rf{STC-I}$_2$, we get
$$\ba{ll}\dm U(t)\dm_{L^\infty(\OO\cap B_R)}\hskip-0.2cm&\displ\leq \dm U_0\dm_{L^\infty(\OO\cap B_R)}+\intll0t \dm U_\tau(\tau)\dm_\infty d\tau\\&\leq (c(R) +t^{\frac12-\frac n{2p}})\dm \n U_0\dm_p\,,\ea$$ where in the last step we estimate the $\dm U_0\dm_{L^\infty(\OO\cap B_R)}$ by $\dm U_0\dm_{W^{1,p}(\OO\cap B_R)}$, and,   taking into account that $U_0=0$ on $\R^n-\OO$\,, by applying the Poincar\é inequality to $\dm U_0\dm_{L^p(\OO\cap B_R)}$\,. The case of $p=n$ is a bit different. Initially we estimate $\dm U(1)\dm_{L^\infty(\OO\cap B_R)}$. By Sobolev inequality, for some $p>n$, $$\dm U(t)\dm_\infty\leq c \dm U(t)\dm_{W^{2,n}(\OO\cap B_R)}.$$ Employing \rf{STC-I}$_{1,2}$, we get
$$\dm \n U(t)\dm_n+\dm D^2 U(t)\dm_n\leq c(1+t^{-\frac12} )\dm \n U_0\dm_n\,,$$ and $$\dm U(t)\dm_{L^n(\OO\cap B_R)}\leq c\dm U_0\dm_{L^n(\OO\cap B_R)}+\intll0t\dm U_\tau(\tau)\dm_nd\tau\leq c\dm \n U_0\dm_n(1+t^\frac12)\,,$$ where in the last step we employ the Poincar\é inequality again. Therefore we can claim that $\dm U(1)\dm_{L^\infty(\OO\cap B_R)}\leq c\dm \n U_0\dm_n$. Finally, in order to obtain \rf{CCP-I}$_2$ in complete form, it is enough  to consider a path with end point 1 for $t<1$ and initial point $1$ for $t>1$, in any case we get $$\dm U(t)\dm_{L^{\infty}(\OO\cap B_R)}\leq \dm U(1)\dm_\infty+ \intll1t\dm U_\tau(\tau)\dm_\infty d\tau\leq \dm \n U_0\dm_n(1+\intll 1t\log\tau d\tau)\,.$$  In order to prove \rf{CCPI-I}$_1$, we remark that for $s<n$   and $r=\infty$ estimate \rf{LI-I} holds,  hence we have  the thesis. In the case of $s>n$, we remark that   $U=0$ on $\{0\}\times\po$, hence we can compute in the following way:
$$\dm U(t)\dm_{L^\infty(\po)}\leq \intll0t\dm U_\tau(\tau)\dm_\infty d\tau \,,$$ that, via \rf{STC-I}$_2$,  implies the thesis. Finally, we consider the case of $s=n$. We repeat the same arguments of the previous case but working in $\OO^c:=\R^n-\ov\OO$. First of all we note that for all $t>0$ it holds $U(t,x)\in C(\OO^c)$, and, by virtue of \rf{GN}, we get $$\dm U(t)\dm_{C(\OO^c)}\leq c\dm D^2U(t)\dm_{L^n(\R^n-\ov\OO)}^\frac12\dm U(t)\dm_{L^n(\R^n-\ov\OO)}^\frac12+c_0\dm U(t)\dm_{L^n(\R^n-\ov\OO)}^{\null}\,,$$ with $c$ independent of $t$ and $U_0$. Moreover, since $\dm U_0\dm_{L^n(\R^n-\OO)}=0$, employing \rf{STC-I}$_{2}$, we have $$\dm U(t)\dm_{L^n(\R^n-\ov{\OO})}\leq \intll 0t\frac d{d\tau}\dm U(\tau)\dm_{L^n(\R^n-\ov\OO)}d\tau\leq c\dm \n U_0\dm_nt^\frac12\,,\;t>0.$$  Therefore, we obtain \be\label{BIU}\dm U(t)\dm_{L^{\infty}(\po)}\leq \dm U(t)\dm_{C(\OO^c)}\leq c\dm \n U_0\dm_n+c_0c\dm \n U_0\dm_nt^\frac12\,,\;t>0\,.\ee Considering the following inequality:$$\dm U(t)\dm_{L^\infty(\po)}\leq \dm U(1)\dm_{L^\infty(\po)}+\intll1t\dm U_\tau(\tau)\dm_{L^\infty(\po)}d\tau\,, $$ via \rf{STC-I}$_2$ for the integral term, and via   estimate  \rf{BIU} for $\dm U(1)\dm_{L^{\infty}(\po)}$, we complete the proof. In order to prove \rf{CCPI-I}$_2$ we consider the following formula:
$$\dm U(t)\dm_{L^p(\po)}\leq \intll0t\dm U_\tau(\tau)\dm_{L^p(\po)}d\tau\,,$$ that by means of \rf{STC-I}$_2$ gives $$\dm U(t)\dm_{L^p(\po)}\leq c\dm \n U_0\dm_pt^\frac12\,.$$ In order to complete the estimates of the lemma we have to prove the one related to $W^{-\frac1q,q}(\po)$ norm. To this end, it is enough to observe that, by the regularity of $U(t)$ and $\po$ bounded, we get   $\dm U(t)\cdot \nu\dm_{-\frac1q,q}\leq c\dm U(t)\dm_\infty$. The same holds in the case of $U_t$. Hence the estimates are a consequence of \rf{CCP-I}$_1$ and \rf{CCPI-I}$_1$, and of \rf{STC-I}$_2$, respectively. Finally, estimate \rf{DSDT} is a consequence of the regularity of $U_0$ and the representation formula.  \end{proof}\section{\label{uS}A special auxiliary Stokes initial boundary value problem}
Let $v_0\in \mathscr C_0(\OO)$. Denoted by $(v,\pi_v)$ and by $(U, c)$ the solutions to problems \rf{STI} and \rf{STC} both with initial data $v_0$, whose existence are ensured by Lemma\,\ref{LI} and by Lemma\,\ref{LSTC}, respectively. The pair $(u,\pi_u)$ with $u:=v-U$ and $\pi_u:=\pi_v-c(t)$ is a solution to the problem
\be\label{STP}\ba{ll}u_t-\Delta u=-\nabla
\pi_u,\;\nabla\cdot u=0,&\mbox{ in
}(0,T)\times\OO,\vspace{3pt}\\u=-U&\mbox{
on
}(0,T)\times\po,\vspace{3pt}\\
u =0&\mbox{ on }\{0\}\times\OO. \ea\ee Trivially, we get   
\be\label{ZF}\intl{\po}u\cdot\nu d\sigma=0\quad\mbox{ and }\quad \intl{\po}u_t\cdot \nu d\sigma=0\,\quad \mbox{ for all }t>0\,.\ee
\begin{lemma}\label{DuL}{\sl Let $p\in(1,\infty)$,  and $q\geq p$. Set $\mu:=\frac n2\left(\frac1p-\frac1q\right)$, for solution $u$  to problem \rf{STP} the following estimates hold:
\be\label{DuL-I}\ba{lll}&\dm \n u(t)\dm_q \leq ct^{-\mu}\dm \nabla v_0\dm_p\,,\;&\mbox{ for   }t\in(0,1)\,,\VS &\dm u_t(t)\dm_q\leq ct^{-\frac12-\mu}\dm \n  v_0\dm_p \,,\;&\mbox{ for   }t\in(0,1)\,,\VS &\dm u(t)\dm_q\leq c t^{\frac12-\mu}\dm \n v_0\dm_p\,,\;&\mbox{ for  }t\in(0,1)\,,\ea\ee with $c$ independent of $v_0$\,.}\end{lemma}
\begin{proof}We set $\widehat \varphi(\tau,x):=\varphi(t-\tau,x)$ for all $\tau\in[0,t]$, where $t>0$ is fixed, and $\varphi_0\in\mathscr C_0(\OO)$.  Taking into account   problem \rf{STP}, in the case of the pair $(u_t,\pi_{u_t})$ and $(\widehat\varphi,\pi_{\widehat\varphi})$  the Green identity \rf{GIF} becomes:
\be\label{DuL-II}(u_t(t),\varphi_0)  =(u_t(s),\varphi(t-s))+\intll st(U_\tau,\nu\cdot T(\widehat\varphi,\pi_{\widehat\varphi}))d\tau\,.\ee  Applying the divergence theorem, recalling that $\widehat \varphi$ is solution to the adjoint problem on $(0,t)\times\OO$,  we get $$(u_t(t),\varphi_0)  =(u_t(s),\varphi(t-s))+\intll st(\n U_\tau,\n\widehat\varphi)d\tau-\intll st(U_\tau,\widehat\varphi_\tau)d\tau\,.$$
A further integration by parts furnishes $$ (u_t(t),\varphi_0)  =(u_t(s),\varphi(t-s))-\intll st(\Delta U_\tau,\widehat\varphi)d\tau-\intll st(U_\tau,\widehat\varphi_\tau)d\tau\,.$$ Hence via \rf{STC}, integrating by parts with respect to the time, we get
\be\label{DTDS}(u_t(t),\varphi_0)=(u_s(s),\varphi(t-s))-(U_t(t),\varphi_0)+(U_s(s),\varphi(t-s))\,.\ee Since $u_s(s)=P\Delta u(s)$, an integration by parts furnishes $$(u_s(s),\varphi(t-s))=(U(s),\nu\cdot \nabla \varphi(t-s)))_{\po}+(u(s),\Delta \varphi(t-s))\,.$$ Letting $s\to0$, we have $u(s,x)\to 0$ in $L^p(\OO)$, as well, recalling \rf{CCPI-I}$_2$, we have $U(s)\to0$ in $L^p(\po)$, then,  we get \be\label{DTDS-I}\lim_{s\to0}(u_s(s),\varphi(t-s))=0\,.\ee
Therefore from \rf{DTDS} and $U_s(s)=\Delta U(s)$, an integration by parts    allows us to deduce
 $$\ba{ll}(u_t(t),\varphi_0)\hskip-0.2cm 
&=(u_s(s),\varphi(t-s))-(U_t(t),\varphi_0)+(U_s(s),\varphi(t-s))\VSE=(  u_s(s), \varphi(t-s))-(U_t(t),\varphi_0)-(\n U(s),\n\varphi(t-s)) \,,\ea$$ and
letting $s\to0$,   we get
\be\label{ut}(u_t(t),\varphi_0)=-(U_t(t),\varphi_0)-(\n v_0,\n \varphi(t)),\mbox{ for   } t>0\,.\ee
Applying Holder's inequality, via estimates \rf{JPA}$_2$ for $\varphi$ and \rf{STC-I}$_2$ for $U$, we obtain
\be\label{RI}\ba{ll}|(u_t(t),\varphi_0)|\hskip-0.2cm&\leq \dm U_t(t)\dm_q\dm \varphi_0\dm_{q'}+\dm  \n v_0\dm_p\dm\n \varphi(t)\dm_{p'}\VSE\leq c\dm \n v_0\dm_pt^{-\frac12-\mu}\dm \varphi_0\dm_{q'}\,,\mbox{ for   } t\in(0,1)\,.\ea\ee Recalling \rf{ZF}, by means of estimate \rf{CLM-I} we also obtain
$$\dm u_t(t)\dm_q\leq c(\dm \n v_0\dm_pt^{-\frac12-\mu}+\dm u_t\cdot\nu\dm_{-\frac1q,q})\leq c(\dm \n v_0\dm_pt^{-\frac12-\mu}+\dm U_t\cdot\nu\dm_\infty)$$ which implies \rf{DuL-I}$_2$ after applying \rf{CCPI-I}$_3$ for $U_t$\,.
As a consequence we also prove that 
\be\label{SLM-I}\dm u(t)\dm_p\leq \intll0t\dm u_\tau(\tau)\dm_pd\tau\leq ct^{\frac12}\dm \n v_0\dm_p\,, \mbox{ for   } t\in(0,1)\,,\ee which proves \rf{DuL-I}$_3$ for $q=p$. Since for all $t>0$, $\dm v(t)\dm_q\leq c\dm v_0\dm_pt^{-\mu}$, via \rf{CCP-I}, we obtain that for all $t>0$ and $R>0$ the estimate $\dm u(t)\dm_{L^q(\OO\cap B_R)}<\infty$ holds. Fixing $t$ in \rf{STI}, by Lemma\,\ref{AES} we get
$$\dm D^2v\dm_q\leq c(\dm v_t\dm_q+\dm v\dm_{L^q(\OO\cap B_R)}),\mbox{ for   } t\in(0,1)\,.$$ Hence the following holds   $$ \dm D^2u\dm_q \leq c(\dm u_t\dm_q+\dm U_t\dm_q+\dm u\dm_{L^q(\OO\cap B_R)}+\dm U\dm_{L^q(\OO\cap B_R)})  ,\mbox{ for   } t\in(0,1)\,.$$ Since \rf{GN-I} for all $\delta\in(0,1)$ furnishes $\dm u\dm_q\leq \delta\dm D^2u\dm_q+c(\delta)\dm u\dm_p$, by virtue of estimates \rf{STC-I}$_{2}$ and \rf{SLM-I}\,, and by virtue of estimate \rf{DuL-I}$_2$\,, for a suitable $\delta$, we get
\be\label{DuL-IIV}\ba{ll}\dm D^2u\dm_q\hskip-0.2cm&\leq  c(\dm u_t\dm_q+\dm U_t\dm_q+\dm u\dm_{L^p(\OO\cap B_R)}+\dm U\dm_{L^q(\OO\cap B_R)})\VS  &\leq ct^{-\frac12-\mu}\dm \n v_0\dm_p\,,\mbox{ for }t\in(0,1)\,,\ea\ee where the constant $c$ is independent of $v_0$ and $t$\,.
 Employing estimate \rf{CM} of Lemma\,\ref{ICM} we deduce $$\dm \n u(t)\dm_q\leq c\dm D^2u(t)\dm_q^a\dm u(t)\dm_p^{1-a}\,\mbox{ for }t\in(0,1)\,.$$ Hence    estimate \rf{DuL-I}$_1$ follows by means of   \rf{SLM-I}-\rf{DuL-IIV}. Employing again estimate \rf{CM},   for all $q>p$ we deduce that $\dm u(t)\dm_q\leq c\dm \n u(t)\dm_q^a\dm u(t)\dm_p^{1-a}$, which  completes the proof via \rf{DuL-I}$_1$ and \rf{SLM-I}.
\end{proof} \begin{lemma}\label{utL}{\sl Let $\OO$ be an exterior exterior and $p\in(1,\infty)$. For $q\geq p$ we set  $\mu=\frac n2\left(\frac1p-\frac1q\right)$. Then for the time derivate of solution $u$ to problem \rf{STP} the following estimate holds:
\be\label{ut-I}  \dm u_t(t)\dm_q\leq ct^{-\frac12-\mu}\dm \n v_0\dm_p \,, \mbox{for }t>1\,,   \ee where  $c$ is a constant independent of $v_0$.  }\end{lemma}
\begin{proof} By virtue of the semigroup property \rf{JPA}$_1$ for $v_t$, that is $\dm v_t(t)\dm_q\leq c(t-s)^{-\mu}\dm v_s(s)\dm_p$, and by virtue of \rf{STC-I}$_2$ for $U_t$, since $v=u+U$ we can limit ourselves to consider the proof for $q=p$, that is $\mu=0$\,\footnote{\;Actually we have$$\ba{ll}\dm u_t(t)\dm_q\hskip-0.2cm&\leq \dm v_t(t)\dm_q+\dm U_t(t)\dm_q\leq ct^{-\mu}\dm v_t(\frac  t2)\dm_p+ct^{-\frac12-\mu}\dm \n v_0\dm_p\\& \leq ct^{-\mu}\Big[\dm u_t(\frac t2)\dm_p\!+\dm U_t(\frac t2)\dm_p\Big]\!+ ct^{-\frac12-\mu}\dm \n v_0\dm_p\leq ct^{-\mu}\dm u_t(\frac t2)\dm_p\!+ ct^{-\frac12-\mu}\dm \n v_0\dm_p\,,\ea$$ where we have employed \rf{STC-I}$_2$ for $U_t$ and \rf{JPA}$_1$ for $v_t$.} . We distinguish the cases: $p\in(1,
\frac n{n-1})$ and $p \geq\frac n{n-1} $. In the latter case we have    $ p'\in(1,n]$. We can deduce estimate  \rf{RI} again, hence  we get
 $$\ba{ll}|(u_t(t),\varphi_0)|\hskip-0.2cm&\leq \dm U_t(t)\dm_p\dm \varphi_0\dm_{p'}+\dm \n v_0\dm_p\dm \n\varphi(t)\dm_{p'}\VSE\leq  c\dm \n v_0\dm_p \dm \varphi_0\dm_{p'}t^{-\frac12} \,,\mbox{ for   } t>1\,.\ea$$ Recalling \rf{ZF}, applying Lemma\,\ref{CLM} we arrive at
\be\label{ut-V}\dm u_t(t)\dm_p\leq c(\dm \n v_0\dm_pt^{-\frac12}+\dm U_t\cdot\nu\dm_{-\frac1p,p})\leq c(\dm \n v_0\dm_pt^{-\frac12}+\dm U_t\cdot\nu\dm_\infty)\ee which,   after applying \rf{CCPI-I}$_3$ for $U_t$\,, implies \rf{ut-I}$_2$.             
Now we consider $p\in(1,\frac n{n-1})$. Hence we have $p'\in(n,\infty)$.  Since  integrating by parts we get $$(u_t(s),\varphi(t-s))=(\Delta u(s),\varphi(t-s))=(U(s),\nu\cdot \n \varphi(t-s))_{\po}+(u(s),\Delta\varphi(t-s))\,,$$      recalling that for $s\to0$ we get both $u(s,x)=v(x,s)-U(x,s)\to 0$ in $J^p(\OO)$ and, by virtue of \rf{CCPI-I}$_2$, $U(s)\to0$ in $L^p(\po)$, via the Green Identity \rf{DuL-II}, we deduce \be\label{utt}(u_t(t),\varphi_0)=\intll 0t(U_\tau,\nu\cdot T(\widehat\varphi,\pi_{\widehat\varphi}))_{\po}d\tau=I_1+I_2 ,\mbox{ for   } t>1\,,\ee where we have set $$I_1:=\intll0{\frac t2}(U_\tau,\nu\cdot T(\widehat\varphi,\pi_{\widehat\varphi}))_{\po}d\tau \mbox{ \;and \;}I_2:=\intll {\frac t2}t(U_\tau,\nu\cdot T(\widehat\varphi,\pi_{\widehat\varphi}))_{\po}d\tau\,.$$ 
Integrating by parts, recalling that, letting $s\to0$, \rf{CCPI-I}$_2$ gives $U(s,x)\to0$ in $L^p(\po)$, we obtain $$I_1=(U(\mbox{$\frac t2$}),\nu\cdot T(\varphi,\pi_\varphi)(\mbox{$\frac t2$}))_\po-\intll0{\frac t2}(U,\nu\cdot T(\varphi_\tau,\pi_{\varphi_\tau})_{\po}d\tau:=I_{11}+I_{12}\,.$$
Applying H\"older's inequality, employing \rf{CCP-I}$_1$ for $U$ and \rf{StTs-I}$_3$ for the stress tensor, we get
$$\ba{ll}|I_{11}|\hskip-0.2cm&\leq c\dm U(\mbox{$\frac t2$})\dm_\infty\dm T(\varphi,\pi_{\varphi})(\mbox{$\frac t2$})\dm_{L^{p'}(\po)}\leq c\dm \n v_0\dm_p\dm \varphi_0\dm_{p'}t^{\frac12-\frac n2}\VSE\leq c\dm \n v_0\dm_p\dm \varphi_0\dm_{p'}t^{-\frac12} \,,\mbox{ for }t>1\,.\ea$$ Applying H\"older's inequality, we get
$$|I_{12}|\leq \intll0{\frac12}\dm U\dm_{L^p(\po)}\dm T(\widehat\varphi_\tau,\pi_{\widehat\varphi_{\tau}})\dm_{L^{p'}(\po)}d\tau +\intll{\frac12}{\frac t2}\dm U\dm_\infty\dm T(\widehat\varphi_\tau,\pi_{\widehat\varphi_{\tau}})\dm_{L^{p'}(\po)}d\tau \,.$$ Employing estimates \rf{CCPI-I}$_2$ and \rf{CCP-I}$_1$  for $U$ in  the first and for the second integral,  respectively, and \rf{StTs-ID}$_3$ for the stress tensor, we get
$$\ba{ll}|I_{12}|\hskip-0.1cm&\displ\leq c\dm \n v_0\dm_p\dm \varphi_0\dm_{p'}\Big[t^{-1-\frac n{2p'\hskip-0.1cm\null}}+\intll{\frac12}{\frac t2}\tau^{\frac12-\frac n{ 2p}}(t-\tau)^{-1-\frac n{2p'\hskip-0.1cm\null}}\hskip0.1cmd\tau\Big]\VSE\leq c\dm \n v_0\dm_p\dm \varphi_0\dm_{p'}t^{-\frac12}\,,\mbox{ for }t>1\,,\ea$$ where we have taken into account that $t>1$.
  Moreover  for $I_2$ we obtain 
$$\ba{ll}|I_2|  \hskip-0.2cm&\leq\displ c\intll{\frac t2}t\dm U_\tau\dm_\infty\dm T(\widehat\varphi, \pi_{\widehat\varphi})\dm_{  p'}d\tau \\&\displ\leq\! c\dm \n v_0\dm_p\dm \varphi_0\dm_{p'}\hskip-0.1cm\intll{\frac t2}t\hskip-0.1cm\tau^{-\frac12\left(1+\frac n{p}\right)}\zeta(t\!-\!\tau)d\tau\leq c\dm \n v_0\dm_p\dm \varphi_0\dm_{p'}t^{-\frac 12}\!, \mbox{\,for\,} t\!>\!1,\ea$$ where, employing estimates \rf{StTs-I}, we  considered $\zeta(\sigma)\!=\!\left\{\!\!\!\ba{ll}\sigma^{-\rho_0 }\hskip-0.1cm&\mbox{if }\sigma\in(0,1)\,,\\  \sigma^{-\frac n{2p'\hskip-0.1cm\null}}\hskip-0.1cm&\mbox{if }\sigma>1 \mbox{ and }p'\geq n.\ea\right.$ Increasing the right hand side of \rf{utt} by means of the estimates related to $I_1$ and the one relative to $I_2$, we get
 $$\ba{ll}|(u_t(t),\varphi_0)|\hskip-0.2cm&\leq \dm U_t(t)\dm_p\dm \varphi_0\dm_{p'}+\dm \n v_0\dm_p\dm \n\varphi(t)\dm_{p'}\VSE\leq  c\dm \n v_0\dm_p \dm \varphi_0\dm_{p'}t^{-\frac12} \,,\mbox{ for   } t>1\,.\ea$$
Recalling \rf{ZF}, via Lemma\,\ref{CLM} we obtain the estimate \rf{ut-V}, that applying \rf{STC-I}$_2$ furnishes 
 estimate \rf{ut-I}   for $q=p$.
\end{proof}
\begin{lemma}{\sl Assume that the Green identity \rf{GIF} holds for the pairs   $(u,\pi_u)$ and $(\widehat\varphi,\pi_{\widehat \varphi})$ solutions respectively to problem \rf{STP}    and to problem \rf{AGPR} with $\varphi_0\in \mathscr C_0(\OO)$. Then, we get \be\label{GIFF}(u(t),\varphi_0) =\displ\intll0t(\nu\!\cdot\! T(\widehat\varphi,\pi_{\widehat\varphi}),U)_{\po} d\tau\,.\ee}\end{lemma}\begin{proof}Recalling \rf{DuL-I}$_3$ and \rf{CCPI-I}$_2$, letting $s\to0$, then  $u(s)=v(s)-U(s)\to0$ in $J^p(\OO)$, as well,   in $ {L^p(\po)}$ follow, respectively. Hence, letting $s\to0$ in the Green identity \rf{GIF}, we arrive at \rf{GIFF}. \end{proof}
 \begin{lemma}\label{EFED}{\sl Let $\OO$ be an  exterior domain  and $p\in(1,n)$, $n\geq2$. Then, for all $q\geq p$,   the solution $u$ to problem \rf{STP} enjoys the following estimates:
\be\label{uQ} \dm u(t)\dm_q\leq c\dm \n v_0\dm_p\left\{\hskip-0.2cm\ba{ll}t^{\frac12-\mu}\,&t>1,   n=2\,,\\t^{-\frac n2\left(\frac1p-\frac1q\right)}\,, &t>1\mbox{ and }q\in\mbox{$[p,n)$}\,,\,n\ne2\,,\\   t^{   -\frac n{2}\left(\frac1p-\frac1n\right) }\,, &t>1\mbox{ and }q= n>3\,,\\ t^{\frac12  -\frac n{2p} }\,, &t>1\mbox{ and }q> n\geq3\,,\\ t^{-\vartheta} ,\,\vartheta\in(0,\frac32\!\big[\frac1p\!-\!\frac13\big]) &t>1 \mbox{ and }q=n=3\,, \ea\right.\ee where constant  $c$ is independent of $v_0$.  }\end{lemma}
\begin{proof}   We recall \rf{GIFF}:
  $$(u(t),\varphi_0) =\displ\intll0t(\nu\!\cdot\! T(\widehat\varphi,\pi_{\widehat\varphi}),U)_{\po} d\tau\,.$$ In order to discuss the last integral we have to distinguish the cases  $n=2,3$ and  $n>3$.  
\par\noindent{\bf n=2.}\;\;     Applying H\"older's inequality, we get
$$|(u(t),\varphi_0)| \leq\displ  c\hskip-0.1cm\intll{t-1}t\dm U\dm_\infty\dm T(\widehat\varphi,\pi_{\widehat\varphi})\dm_{q'} d\tau + c\hskip-0.1cm\intll0{t-1}\dm U\dm_\infty\dm T(\widehat\varphi,\pi_{\widehat\varphi})\dm_{r} d\tau\,.$$ We recall \rf{CCP-I}$_1$ for $U$, and, remarking that the best bound for the latter integral is for $r>2$, via \rf{StTs-I}$_{1,3}$ for $T(\widehat\varphi,\pi_{\widehat\varphi})$, we obtain:
\be\label{GFnD}\ba{ll}|(u(t),\varphi_0)|\hskip-0.3cm& \displ\leq c\dm\n v_0\dm_p\dm \varphi_0\dm_{q'}\Big[\hskip-0.15cm\intll{t-1}t\hskip-0.15cm\tau^{\frac12-\frac1p}(t\!-\!\tau)^{-\rho_0}d\tau+\intll0{t-1}\hskip-0.15cm\tau^{\frac12-\frac1p}(t\!-\!\tau)^{\frac1q-1}d\tau \Big] \\&\displ \leq c\dm\n v_0\dm_p\dm \varphi_0\dm_{q'}t^{\frac12+\frac1q-\frac1p},\mbox{ for }t>1\,. \ea\ee  
 \par\noindent{\bf n=3.}\;\; Partially the argument is   the same of the case $n=2$:
 \be\label{GF} (u(t),\varphi_0)  =\intll{t-1}{t }\hskip-0.1cm(\nu\!\cdot\! T(\widehat\varphi,\pi_{\widehat\varphi}),U)_{\po} d\tau +\!\!\intll0{t-1}\!(\nu\!\cdot\! T(\widehat\varphi,\pi_{\widehat\varphi}),U)_{\po} d\tau   \,,\;  t\!>\!1\,. \ee Applying H\"older's inequality, for all $r>q'$, we get $$|(u(t),\varphi_0)| \displ\leq c\hskip-0.1cm\intll{t-1}t\hskip-0.1cm\dm U \dm_\infty\dm \nu\cdot T(\widehat\varphi,\pi_{\widehat\varphi}) \dm_{{L^{q'}(\po)}}d\tau +c\hskip-0.1cm\intll0{t-1}\hskip-0.1cm\dm U \dm_\infty\dm \nu\cdot T(\widehat\varphi,\pi_{\widehat\varphi}) \dm_{{L^{r}(\po)}}d\tau \,.$$ Recalling   estimates  \rf{StTs-I}$_{1,2}$ for the stress tensor, and  estimate \rf{CCP-I}$_1$ for $U$,    we obtain \be\label{NUT}\ba{ll}|(u(t),\varphi_0)|\hskip-0.2cm& \displ\leq c\dm \n v_0\dm_p\dm \varphi_0\dm_{q'}\!\Big[\intll0{t-1}\!\tau^{\frac12-\frac 3{2p}}(t-\tau)^{-\frac 3{2q'\hskip-0.1cm\null}}\, d\tau\displ+\hskip-0.2cm\intll{t-1}t\!\tau^{\frac12-\frac 3{2p}}(t-\tau)^{-\rho_0}\, d\tau\Big]\\&\leq c\dm \n v_0\dm_p\dm \varphi_0\dm_{q'}\!\left\{\hskip-0.2cm\ba{ll}t^{ -\frac 3{2 }\left(\frac1p-\frac1q\right)} ,&\mbox{if }q\!\in\![p,3)\,,\;t\!>\!1\,, \\t^{\frac12-\frac 3{2p}}\,,&\mbox{if } q>3\,,\;t\!>\!1.\ea\right.\ea\ee  
  \vskip0.15cm \par\noindent {\bf n$>$3. } Again we consider the Green formula   for $u$: \be\label{uQ-I} (u(t),\varphi_0) =\displ\intll0t(\nu\!\cdot\! T(\widehat\varphi,\pi_{\widehat\varphi}),U)_{\po} d\tau =:I_1+I_2+I_3\,,\;t>1\,.\ee Applying H\"older's inequality, we get $$ |I_1|= |\intll0{\frac12}\!(\nu\!\cdot\! T(\widehat\varphi,\pi_{\widehat\varphi}),U)_{\po} d\tau|\leq c
\intll0{\frac12}\dm U\dm_{L^p(\po)}\dm\nu\!\cdot\! T(\widehat\varphi,\pi_{\widehat\varphi})\dm_{L^{r}(\po)}    d\tau\,.
$$By virtue of estimates \rf{StTs-I}$_3$ for the stress tensor, and estimate \rf{CCPI-I}$_2$ for $U$, since $q'\leq p'$ we get
$$\ba{ll}|I_1|\hskip-0.2cm&\displ\leq c\dm \n v_0\dm_p\dm \varphi_0\dm_{q'}\intll0{\frac12}\hskip-0.1cm  (t-\tau)^{-\frac n2\left(1-\frac 1{q}\right)}  d\tau\VSE\leq c\dm \n v_0\dm_p\dm \varphi_0\dm_{q'}t^{-\frac n2\left(\frac1p-\frac1q\right)}\,,\;t>1,\ea$$ where we have employed the assumption $t>1$. Applying H\"older's inequality, for the term $I_2$ we get
$$  |I_2(t)|=|\intll{\frac12}{t-\frac12}( \nu\cdot T(\widehat\varphi,\pi_{\widehat\varphi}),U)_{\po}  d\tau| \leq c \intll{\frac12}{t-\frac12}\dm U(t)\dm_\infty \dm \nu\cdot T(\widehat\varphi,\pi_{\widehat\varphi})\dm_{L^{r}(\po)} d\tau \,.$$
By virtue of estimates \rf{StTs-I}$_3$ for the stress tensor, and estimate \rf{CCP-I}$_1$ for $U$,    we get 
$$\ba{ll}|I_2| \hskip-0.2cm&\displ \leq c\dm \n v_0\dm_p\dm \varphi_0\dm_{q'}\hskip-0.2cm\intll{\frac12}{t-\frac12}\hskip-0.1cm  \tau^{\frac12\left(1-\frac n{p}\right)}  (t-\tau)^{-\frac n2\left(1-\frac1q\right)}   d\tau\VSE\leq c\dm \n v_0\dm_p\dm \varphi_0\dm_{q'} t^{-\frac n2\left(\frac1p-\frac1q\right)} ,\;t>1 \,.\ea$$ Finally, applying H\"older's inequality, we get
$$|I_3|=| \intll{t-\frac12}t \!(\nu\!\cdot\! T(\widehat\varphi,\pi_{\widehat\varphi}),U)_\po d\tau |\leq c\intll{t-\frac12}t\dm U\dm_\infty\dm \nu\cdot T(\widehat\varphi,\pi_{\widehat\varphi})\dm_{q'}d\tau ,\;  t\!>\!1\,. $$ By virtue of estimates \rf{StTs-I}$_1$ for the stress tensor, and estimate \rf{CCP-I}$_1$ for $U$,      we get  $$\ba{ll}|I_3|\hskip-0.2cm&\displ\leq c\dm \n v_0\dm_p\dm \varphi_0\dm_{q'}\intll{t-\frac12}t\tau^{\frac 12\left(1-\frac np\right)}(t-\tau)^{-\rho_0}d\tau\\&\leq c\dm \n v_0\dm_p\dm \varphi_0\dm_{q'}\left\{\hskip-0.2cm\ba{ll}t^{-\frac n2\left(\frac1p-\frac1q\right)}\,,&\mbox{if }q\leq n\,, \\t^{\frac12-\frac n{2p}}\,,&\mbox{if }q>n\,.\ea\right.\ea$$ Collecting the estimates related to $I_1,I_2,I_3$, via the Green formula \rf{uQ-I},  we obtain 
\be\label{NUTI}|(u(t),\varphi_0)|\!\leq \!c\dm \n v_0\dm_p\dm \varphi_0\dm_{q'}\hskip-0.1cm\left\{\hskip-0.2cm\ba{ll}t^{-\frac n2\left(\frac1p-\frac1q\right)},\hskip-0.15cm&\mbox{if }q\!\leq\! n\,, \\t^{\frac12-\frac n{2p}}\,,&\mbox{if }q\!>\!n\,,\ea\right. \mbox{for all }  t\!>\!1\,.\ee We are in a position to prove \rf{uQ}. Since estimates \rf{GFnD}, \rf{NUT} for $q\ne3$, and \rf{NUTI} hold for all $\varphi_0\in \mathscr C_0(\OO)$, recalling \rf{ZF} for $u$,  via Lemma\,\ref{CLM} and estimates \rf{CCPI-I}$_1$, for the norm $\dm u\cdot\nu\dm_{-\frac1q,q}\equiv \dm U\cdot\nu\dm_{-\frac1q,q}$,  one
 proves \rf{uQ} in all the cases with exclusion of \rf{uQ}$_5$.  
For this last, employing \rf{uQ}$_2$ and \rf{uQ}$_4$ for $n=3$, we prove \rf{uQ}$_5$   interpolating $\dm u(t)\dm_3$ between $q_1>3$ and $q_2=3-\eta>p$. Hence we get
$$\dm u(t)\dm_3\leq \dm u(t)\dm_{q_1}^a\dm u(t)\dm_{q_2}^{1-a}\leq c\dm \n v_0\dm_p t^{-\vartheta},\;t>1,$$
where we set $\vartheta:=\frac3{2p}-\frac a2-\frac12 $ with $a:=\frac{\eta q_1}{3(q_1-3+\eta)}$, that proves the result for $q_1$ sufficiently large and $\eta$ sufficiently small. The lemma is completely proved.
 \end{proof} \begin{lemma}\label{Up}{\sl Let $\OO$ be an exterior domain and $q\geq p\geq n,\, n\geq2\,.$   Then for the solution $u$ to problem \rf{STP} the following estimates hold:
\be\label{Up-I}\dm u(t)\dm_q\leq c\dm \n v_0\dm_p\Gamma(t)\,,\;t>1\,,\ee where we have set $\Gamma(t):= \left\{\ba{ll}t^{\frac12-\mu}&\mbox{if }q\geq p\geq2,\,n=2,\\t^{\frac12-\frac n{2p}}&\mbox{if }q\geq p>n\geq2,\\\log(t+e)&\mbox{if } q\geq p=n,\,n >3,\\\log(t+e)&\mbox{if }q>p=3\hskip0.03cm,\,n=3,\\\log^2(t+e)&\mbox{if }q=p=3\hskip0.03cm,\,n=3.\ea\right.$  } \end{lemma}
\begin{proof}We start   from the Green identity  \rf{GIFF} for $u$:
\be\label{Up-II} (u(t),\varphi_0) =\displ\intll0t(\nu\!\cdot\! T(\widehat\varphi,\pi_{\widehat\varphi}),U)_{\po} d\tau =I_1+I_2 \,,\;t>1\,,\ee where$$\ba{l} I_1:=\displ\intll0{t-1}(\nu\!\cdot\! T(\widehat\varphi,\pi_{\widehat\varphi}),U)_{\po} d\tau\,,\\\displ I_2:=\intll{t-1}t(\nu\!\cdot\! T(\widehat\varphi,\pi_{\widehat\varphi}),U)_{\po} d\tau\,.\ea$$ 
Applying H\"older's inequality, for the term $I_1$ we get: $$|I_1| \leq c\intll0{t-1}\dm   T(\widehat\varphi,\pi_{\widehat\varphi})\dm_{n}\dm U\dm_{L^\infty(\po)}^{\null}  d\tau\,.$$ Employing \rf{CCPI-I}$_1$ for $U$, and, since $q'\leq  p'\leq n'$, employing \rf{StTs-I}$_3$ for the stress tensor, we obtain $$|I_1|\leq c\dm \n v_0\dm_p\dm \varphi_0\dm_{q'}\intll0{t-1}\zeta_b(\tau)(t-\tau)^{ -\frac n2\left(1-\frac1q\right)}d\tau\leq c\dm \n v_0\dm_p\dm \varphi_0\dm_{q'}\Gamma(t)\,.$$ Finally, we estimate $I_2$. Applying H\"older's inequality, employing \rf{CCPI-I}$_1$ for $U$ and \rf{StTs-I}$_1$ for the stress tensor,  recalling Remark\,\ref{RSt},  we get
$$\ba{ll}|I_2|\hskip-0.2cm&\displ\leq c\intll{t-1}t\dm T(\widehat\varphi,\pi_{\widehat\varphi})\dm_{q'}\dm U\dm_{L^\infty(\po)}d\tau\leq c\dm \n v_0\dm_p\dm\varphi_0\dm_{q'}\intll{t-1}t\zeta_b(\tau)(t-\tau)^{-\rho_0}d\tau\VSE\leq c\dm \n v_0\dm_p\dm \varphi_0\dm_{q'}\Gamma(t)\,.\ea$$ Hence we deduce $$|(u(t),\varphi_0)|\leq  c\dm\n v_0\dm_q\dm \varphi_0\dm_{q'}\Gamma(t),\mbox{ for all }\varphi\in \mathscr C_0(\OO)\mbox{ and }t>1\,.$$ Recalling \rf{ZF} for $u$, via estimate \rf{CLM-I} and estimate \rf{CCPI-I}$_1$ one
  completes the proof.  
\end{proof}\begin{lemma}\label{Gu}{\sl Let $\OO$ be an  exterior domain and  $n\geq2$. Then for the solution $u$ to problem \rf{STP}, for $q\ge p$, the following estimate holds:
\be\label{Gu-I}  \dm \n u(t)\dm_q\leq c\dm \n v_0\dm_p\widehat g_p(t)\,,  \ee where the constant  $c$ is independent of $u$ and $\widehat g_p(t)$ is defined by \be\label{SLMIV}\widehat g_p(t):=\hskip-0.1cm\left\{\hskip-0.2cm\ba{lll}   t^{-\mu},&t\in(0,1)\,,&\null\mbox{if }q\geq p>1\,,
\\t^{\frac12-\mu},&t>1 \,,&\mbox{if }n=2\,,\,p\ne2\,,
\\\log^\frac32(t+e)\,,&t\geq1\,,& \mbox{if }q=p=n=3\,,\\\log (t+e)\,,&t\geq1\,, &\hskip-0.25cm\left\{\hskip-0.2cm\ba{l}\mbox{if }q>p=n=3\,,\\\mbox{if }q\geq p=n>3\,, \ea\right. \\ t^{-\mu},& t>0\,,&\mbox{if }q\in[p,n),n>2 \\ t^{\frac12-\frac n{2p}}\,,&t\geq1\,,&\mbox{if }q\geq p\ne n>3   \,   .  \ea\right.\ee  }\end{lemma}\begin{proof}The estimates for $t\in(0,1)$ are contained in \rf{DuL-I}$_1$. Hence we limit ourselves to look for the estimates for  $t\geq1$. Since equation \rf{STI}$_1$ ensures $v_t=P\Delta v$, by virtue of Lemma\,\ref{AES}   we get
\be\label{DSQ}\dm D^2v(t)\dm _q\leq c(\dm v_t(t)\dm_q+\dm v(t)\dm_{L^q(\OO')})\,.\ee Since $v=U+u$, and $\OO'$ is bounded, for all $r>q$, we deduce
\be\label{DSSLM}\dm D^2 u(t)\dm_q\leq \dm D^2U(t)\dm_q+c(\dm U_t\dm _q+\dm u_t\dm_q+\dm U(t)\dm_{L^\infty(\OO\cap B_R)}+\dm u(t)\dm_{L^r(\OO\cap B_R)})\,.\ee Hence via Lemma\,\ref{LSTC} and Corollary\,\ref{CCP} for $U$, Lemma\,\ref{utL}  for $u_t$, Lemma\,\ref{EFED} - Lemma\,\ref{Up} for $u$, for $q\geq p\ne2$ and    $n=2$, we get
\be\label{SLM} \dm D^2u(t)\dm_q\leq c\dm\n v_0\dm_pt^{\frac12+\mu-\frac1p},\,\mu>0, \;t>1\,,\ee  and, for $q\geq p$ and $n\ne2$, we get \be\label{SLMI}\ba{l}\dm D^2u(t)\dm_q\leq c\dm \n v_0\dm_pt^{-\frac n2\left(\frac1p-\frac1q\right)},\mbox{ if  }q\in[p,n)\,,\quad t>1\,, \\\dm D^2u(t)\dm_q\leq c\dm \n v_0\dm_p\left\{\hskip-0.2cm\ba{ll}t^{\frac12-\frac n{2p}}, &\mbox{if }q >n\,,\\ \log(t+e),&\mbox{if }q=n\,,\ea\right.\;\;\, t>1\,.\ea\ee  
Employing Lemma\,\ref{ICM}, for all $q\in(1,\infty)$, we obtain
$$\dm \n u(t)\dm_q\leq c\dm D^2u(t)\dm_q^\frac12\dm u(t)\dm_q^\frac12\,,\;t>1\,.$$ Hence,     via estimate    \rf{SLMI} (resp. \rf{SLM} for $n=2$) for $D^2u$, and via estimates  \rf{uQ} and \rf{Up-I} for $u$, we get \rf{Gu-I} with $\widehat g_p$ given by \rf{SLMIV}.  \end{proof}
 \section{\label{CSTF}Some consequences of the results of Section\,\ref{US} and Section\,\ref{uS}} In this section, we assume $v_0\in \mathscr C_0(\OO)$ and we establish some properties of the solutions to problema \rf{STI} whose existence is ensured by Corollary\,\ref{CSP}.   
\begin{lemma}\label{CGLP}{\sl Let $v_0\in\mathscr C_0(\OO)$ and $(v,\pi_v)$  be the solution of Corollary\,\ref{CSP}, then $\nabla v\in C([0,T);L^p(\OO))$ holds with $\lim_{t\to0}\dm \nabla v(t)-\nabla v_0\dm_p=0$.}\end{lemma}\begin{proof}   By virtue of Lemma\,\ref{LSTC}, $\nabla U\in C([0,T);L^p(\R^n)$ with $\lim_{t\to0}\dm \nabla U(t)-\nabla v_0\dm_p=0$. Since $v=U+u$, the result   is achieved if we are able to prove that $\nabla u\in C([0,T);L^p(\OO))$ and $\lim_{t\to0}\dm \nabla u(t)\dm_p=0$.   From formula \rf{DTDS}, applying H\"older's inequality, via estimate \rf{DSDT},  
we get
$$|(u_t(t),\varphi_0)|\leq |(v_s(s),\varphi(t-s))|+ c\dm D^2v_0\dm_p(\dm \varphi_0\dm_{p'}+\dm \varphi(t-s)\dm_{p'})\,,\;t>0\,.$$ Then the limit property \rf{DTDS-I} fot $U_t$   and estimate \rf{JPA}$_1$ for $\varphi$ furnish $$\dm u_t(t)\dm_p\leq c\dm D^2 U_0\dm_p\,,\mbox{ for all }t>0\,.$$ Via the Minkowski inequality, employing again \rf{DSDT}, easily it holds $$\dm   v_t\dm_p\leq \dm U_t\dm_p+\dm u_t\dm_p\leq c\dm D^2 v_0\dm_p\,\mbox{ for all }t>0\,. $$
Since equation \rf{STI}$_1$ ensures $v_t=P\Delta v$, by virtue of Lemma\,\ref{AES}   we get
$$\dm D^2v(t)\dm _q\leq c(\dm v_t(t)\dm_q+\dm v(t)\dm_{L^q(\OO')})\,.$$ Hence, via \rf{JPA}$_1$ for $\dm v(t)\dm_p$, the following estimate holds $$\dm D^2v(t)\dm_p\leq c(\dm D^2 v_0\dm_p+\dm v_0\dm_p).$$ As well, applying the Minkowski inequality and again \rf{DSDT}, we get
\be\label{DSDT-II}\dm D^2 u(t)\dm_p\leq c(\dm D^2 v_0\dm_p+\dm v_0\dm_p)\,\mbox{  for all }t>0\,.\ee
 Via inequality \rf{CM}, for all $t$ and $s$ we obtain $$\dm\nabla v(t)-\nabla v(s)\dm_p\!\leq c\dm D^2 v(t)-D^2 v(s)\dm_p^\frac12\dm v(t)-v(s)\dm_p^\frac12\!\leq c(\dm D^2 v_0\dm_p+\dm v_0\dm_p)^\frac12\dm v(t)-v(s)\dm_p^\frac12\,,$$ that, via Corollary\,\ref{CSP}, furnishes the continuity, and as well the one of $\nabla u(t)$ holds. Applying inequality \rf{CM}, we obtain $$\dm \nabla u(t)\dm_p\leq c\dm D^2u(t)\dm_p^\frac12\dm u(t)\dm_p^\frac12\,,\mbox{ for all }t>0\,.$$ Hence the limit property for $\dm \nabla u(t)\dm_p$ follows from \rf{DSDT-II} and \rf{DuL-I}$_3$ for $\dm u(t)\dm_p$.   
\end{proof}
\begin{lemma}\label{TV}{\sl Let $(v,\pi_v)$  be the solution of Corollary\,\ref{CSP}, then the following estimates holds:
\be\label{TV-I}\dm v_t(t)\dm_q\leq c\dm \n v_0\dm_p t^{-\frac12+\frac n2\left(\frac1p-\frac1q\right)},\quad t>0\,,\ee where $c$ is a constant independent of $v$.}\end{lemma}\begin{proof} Estimate \rf{TV-I} is an immediate consequence of estimates \rf{ut-I} and \rf{CCPI-I}$_2$.\end{proof}
\begin{lemma}\label{DuLv}{\sl Let $p\in(1,\infty)$,  and $q\geq p$. Set $\mu:=\frac n2\left(\frac1p-\frac1q\right)$, for solution $(v,\pi_v)$  furnished by Corollary\,\ref{CSP} enjoys  the following estimates hold:
\be\label{DuLv-I}\dm \n v(t)\dm_q \leq ct^{-\mu}\dm \nabla v_0\dm_p\,,\;\mbox{ for   }t\in(0,1)\,, \ee with $c$ independent of $v$\,.}\end{lemma}\begin{proof} We consider $v=U+u$. Hence the result is a consequence of Lemma\,\ref{LSTC} for $\n U$ and of Lemma\,\ref{DuL} for $\n u$.\end{proof} 
\begin{lemma}\label{Guv}{\sl Let $\OO$ be an  exterior domain and  $n\geq2$. Then for the solution $(v,\pi_v)$ furnished by Corollary\,\ref{CSP} enjoys the following estimates hold:
\be\label{Guv-I}  \dm \n v(t)\dm_q\leq c\dm \n v_0\dm_p\widehat g_p(t)\,,\mbox{ for }t>1,\mbox{ and }q\geq p\,, \ee where the constant  $c$ is independent of $v$ and $\widehat g_p(t)$ is defined by \be\label{SLMIVV}\widehat g_p(t):=\hskip-0.1cm\left\{\hskip-0.2cm\ba{lll}   t^{-\mu},&t\in(0,1)\,,&\null\mbox{if }q\geq p>1\,,
\\t^{\frac12-\mu},&t>1 \,,&\mbox{if }n=2\,,\,p\ne2\,,
\\\log^\frac32(t+e)\,,&t\geq1\,,& \mbox{if }q=p=n=3\,,\\\log (t+e)\,,&t\geq1\,, &\hskip-0.25cm\left\{\hskip-0.2cm\ba{l}\mbox{if }q>p=n=3\,,\\\mbox{if }q\geq p=n>3\,, \ea\right. \\ t^{-\mu},& t>0\,,&\mbox{if }q\in[p,n),n>2 \\ t^{\frac12-\frac n{2p}}\,,&t\geq1\,,&\mbox{if }q\geq p\ne n>3   \,   .  \ea\right.\ee  }\end{lemma}
\begin{proof}
We consider $v=U+u$. Hence the result is a consequence of Lemma\,\ref{LSTC} for $\n U$ and Lemma\,\ref{Gu} for $\n u$.
\end{proof}
 \begin{lemma}\label{LUSLM}{\sl Let $\OO$ be an exterior domain and $n=2$.   Then, for $q\geq2$ the solution $(v,\pi_v)$ of Corollary\,\ref{CSP} is such that
\be\label{USLM-O}\dm \n v(t)\dm_q\leq \widehat g_2(t)\dm \n v_0\dm_2,\;t>1,\quad \widehat g_2(t):=\left\{\hskip-0.2cm\ba{ll}1&\mbox{if }q=2\,\VS c>1&\mbox{if }q>2\,,\ea\right.\ee where $c$ is a constant independent of     $v$.
  }\end{lemma} \begin{proof}  In order to prove \rf{USLM-O}, by virtue of Lemma\,\ref{AES},   
employing the Poincar\é inequality, we easily get
$$\dm D^2 v(t)\dm_2\leq c(\dm v_t(t)\dm_2+\dm  v(t)\dm_{L^2(\OO')})\leq  c(\dm v_t(t)\dm_2+\dm \n v(t)\dm_2).$$ Since the $L^2$-theory ensures  $\dm \n v(t)\dm_2\leq \dm \n v_0\dm_2,t>0$, for all $v_0\in \mathscr C_0(\OO)$, that is \rf{USLM} with $\widehat g_2=1$,   employing \rf{TV-I}, we arrive at 
\be\label{FDSD}\dm D^2v(t)\dm_2\leq c\dm \n v_0\dm_2(1+t^{-\frac12})\,,t>1\,.\ee 
By virtue of Lemma\,\ref{ICM} we get
$$\dm \n v(t)\dm_q\leq c\dm D^2v(t)\dm_2^a\dm \n v(t)\dm_2^{1-a}\,,\;t>1\,,\;\mbox{$a:=\frac{q-2}q$}.$$
Employing \rf{FDSD}, and employing again $\dm \n v(t)\dm_2\leq\dm\n v_0\dm_2,$ we conclude the proof with $\widehat g_2\equiv c$.
\end{proof}\begin{lemma}\label{UUSLM}{\sl 
Let $\OO$ be an exterior domain and $n>2$. In \rf{STI} assume   $v_0\in\mathscr C_0(\OO)$. Then, for $q\geq n,\,p\in(1,n)$ the solution $(v,\pi_v)$ of Corollary\,\ref{CSP} is such that
\be\label{USLM}\dm \n v(t)\dm_q\leq c\widehat g_p(t)\dm \n v_0\dm_p,\;t\!\geq\!1,\; \widehat g_p(t):=\!\left\{\hskip-0.2cm\ba{ll} t^{-\frac n2\left(\frac1p-\frac1n\right)},\hskip-0.2cm&\mbox{if }q>3\,,\VS t^{\frac14-\frac 3{4p}-\frac\vartheta2},&\mbox{if }q=3\,,\ea\right.\ee where $c$ is a constant independent of $v$.
  }\end{lemma}\begin{proof}     We look for an estimate for $\dm\n u(t)\dm_q$. In our hypotheses we can deduce again \rf{DSSLM}, that is, for all $r>q>p$, 
\be\label{UDSQ}\dm D^2u(t)\dm _q\leq \dm D^2U(t)\dm_q+ c(\dm v_t(t)\dm_q+\dm U(t)\dm_{L^\infty(\OO')}+\dm u(t)\dm_{L^r(\OO')})\,.\ee  
By virtue of estimates \rf{TV-I}, \rf{CCP-I}$_1$ and \rf{uQ}$_4$, we get
\be\label{DDSSLM}\dm D^2u(t)\dm_n \leq c\dm \n v_0\dm_pt^{\frac12-\frac n{2p}} ,\;t>1\,.\ee Since estimate \rf{ICM}, furnishes $\dm \n u\dm_q\leq c\dm D^2 u(t)\dm_q^\frac12\dm  u(t)\dm_q^\frac12$,  applying \rf{DDSSLM} and \rf{uQ}$_3$ for $q=n>3$ and \rf{uQ}$_5$ for $q=n=3$, we get
$$\dm  \n u(t)\dm_n\leq c \dm  \n v_0 \dm_p\left\{\hskip-0.2cm\ba{ll}t^{\frac 12-\frac n{2p}},&\mbox{if }q=n>3,\;t>1\,,\VS t^{\frac14-\frac n{4p}-\frac\vartheta2},&\mbox{if }q=n=3,\;t>1\,,\ea\right.$$
moreover, applying \rf{DDSSLM} and \rf{uQ}$_4,$ for $q>n$, we get
$$\dm  \n u(t)\dm_q\leq c \dm  \n v_0 \dm_pt^{\frac 12-\frac n{2p}},\mbox{ if }n\geq3,\;t>1\,, $$
Since $v=U+u$, recalling \rf{STC-I}$_1$, we deduce the result \rf{USLM} via the  Minkowski inequality. 
   \end{proof}
\section{\label{PCT}Proof of Theorem\,\ref{CT}, Theorem\,\ref{CTE} and Proposition\,\ref{CCT}}
\subsection{{\it Proof of Theorem\,\ref{CT}}.}\par 
Let $v_0\in J_0^p(\OO)$. Since $\OO$ is bounded, $v_0\in J^{1,p}(\OO)$. Moreover, there exists a sequence  $\{v_0^k\}\subset \mathscr C_0(\OO)$ which converges to $v_0$ in $J^{1,p}(\OO)$. We denote by $(v_k,\pi_{v_k})$ the sequence of solutions  ensured by Lemma\,\ref{LI} and enjoying  the estimates \rf{BDC-I}. We consider the decomposition $v^k:=U^k+u^k$. Using the linearity of the Stokes problem, by virtue of estimate \rf{STC-I}$_1$ for $\{U^k\}$ and estimate  \rf{DuL-I}$_1$ for $\{u^k\}$,  we get $\dm \n v^p(t)-\n v^m(t)\dm_q\leq ct^{-\mu}\dm \n v_0^p-\n v_0^m\dm _p$, for $t\in (0,1)$. For $t\geq1$ we employ estimates \rf{BDC-I},  hence $\dm \n v^k(t)-\n v^m(t)\dm_q\leq ct^{-\frac n2\left(\frac1p-\frac1q\right)}\exp[-\gamma t]\dm v_0^k- v_0^m\dm_p$, that we increase via the Poincar\`{e} inequality. An analogous argument is developed   in the case of $\{v_t^k\}$.  We employ \rf{STC-I}$_2$ for $\{U_t^k\}$ and estimate   \rf{DuL-I}$_2$ for $\{u_t^k\}$, provided that $t\in(0,1)$. In the case of $t\geq 1$, we employ \rf{BDC-I}  for $\{u_t^k\}$. Hence for the linearity of the problem \rf{STI} we get that $\dm v_t^k-v_t^m\dm_q \leq ct^{-\frac12-\mu}\dm \n v^k_0-\n v^m_0\dm_p$. For $\{D^2v^k\}$ and $\{\pi_{v^k}\}$ we employ Lemma\,\ref{AES}. Then, for all $t>0$, the sequence $\{v^k\}$ enjoys  the Cauchy condition. We denote by $(v,\pi_v)$ the limit.    Since for $q=p$ the above Cauchy conditions for $\{\n v^k\}$ are uniform with respect to $t$, and by virtue of  Lemma\,\ref{CGLP}   $\{v^k\}\subset C([0,T);J_0^p(\OO))$ with $\lim_{t\to0}\dm \n v^k(t)-\n v^k_0\dm_p=0$, we get   that the limit  $v\in C([0,T);J_0^p(\OO))$ and $\lim_{t\to0}\dm \n v(t)-\n v_0\dm_p=0 $. Therefore     the limit $(v,\pi_v) $ of the sequence $\{v^k\}$   enjoys the estimates \rf{CT-oi}-\rf{CT-ii}. The uniqueness holds as in the case of the usual $L^p$-theory. \chiu
\subsection{{\it Proof of Theorem\,\ref{CTE}}.}\par {\sc Existence}. 
  Let $v_0\in J^p_0(\OO)$. We denote by $\{v_0^k\}\subset \mathscr C_0(\OO)$ a sequence converging to $v_0$ in $J^p_0(\OO)$. By virtue of Corollary\,\ref{CSP}, we denote by $\{(v^k,\pi_{v^k})\}$ the sequence of solutions to problem \rf{STI}. We also set $v^k:=U^k+u^k$ and $\pi_{v^k}:=\pi_{u^k}$. Hence, by virtue of the linearity of problem \rf{STI},   employing  Lemma\,\ref{DuLv}, Lemma\,\ref{Guv}, Lemma\,\ref{LUSLM} and Lemma\,\ref{UUSLM}, we get      $$\dm \n v^k(t)-\!\n v^m(t)\dm_q\leq g_p(t)\dm \n v_0^k-\n v^m_0\dm_p,\mbox{ for all }k,m\in \N\mbox{ and }t>0\,,$$ where, collecting the estimates given in Lemma\,\ref{DuLv}, Lemma\,\ref{Guv}, Lemma\,\ref{LUSLM} and Lemma\,\ref{UUSLM},  
we tacitly defined $g_p(t)$ as made in the statement of Theorem\,\ref{CT}.  Moreover, by virtue of Lemma\,\ref{TV}, for $\mu:=\frac n2(\frac 1p-\frac1q)$, we get $$\dm v^k_t(t)-v^m_t(t)\dm_q\leq ct^{-\frac12-\mu}\dm \n v_0^k-\n v_0^m\dm_p,\mbox{ for all }k,m\in \N\mbox{ and }t>0\,.$$ Finally, employing Lemma\,\ref{AES}, via the above estimates for $v^k-v^m$ and the Poincar\é inequality, we get$$ \dm D^{ 2} v^k(t)-\!D^{ 2} v^m(t)\dm_q+\dm\n \pi_{v^k}-\n\pi_{v^m}\dm_q\!\leq cg_p(t)\dm \n v_0^k\!-\!\n v^m_0\dm_p,\mbox{ for all }k,m\!\in\! \N\mbox{ and }t>0\,.$$
  Since the right  hand side of the above estimates satisfies the Cauchy condition in $J_0^p(\OO)$, we get the existence of  strong limit  $(v,\pi_v)$ solutions to problems \rf{STI}.    Since for $q=p$ the above Cauchy conditions for $\{\n v^k\}$ are uniform with respect to $t$ on any compact interval $[0,T]$, as proved  in the case of $\OO$ bounded, we get that the limit $v\in C([0,T);J_0^p(\OO))$ and $v(t,x)$ assume the initial data $v_0(x)$ by continuity in the norm of $J_0^p(\OO)$. The pair $(v,\pi_v)$   is a solution to problem \rf{STI} and enjoys property \rf{CT-i}-\rf{CT-iii},   the proof of the existence is completed. 
\par {\sc Uniqueness}.  We prove that in the class of existence for $v_0$ the uniquee solution is identically equal to 0.   Since $\nabla v\in C([0,T;J_0^p(\OO))$, for all $t>0$ and $R>0$, via the Poincar\é inequality, we get $$v\in C([0,T);L^p(\OO\cap B_R))\,.$$ Since $v_t\in L^1(0,T;L^p(\OO))$, $v=0$ in $t=0$, for all $t>0$ the following also holds:
$$\dm v(t)\dm_{L^p(\OO\cap B_R)}\leq \intll0t\dm v_\tau(\tau)\dm_{L^p(\OO\cap B_R)}d\tau\leq \intll 0t\dm v_\tau(\tau)\dm_pd\tau\,.$$ The last inequality holds uniformly in $R>0$. Hence letting $R\to\infty$, we get  $v\in L^\infty(0,T;L^p(\OO))$. Now, the uniqueness follows by the one of the usual $L^q$-theory.\chiu  \subsection{\it Proof of Proposition\,\ref{CCT}.} \hskip0.53cm  
We start proving point i.\,. We can assume $v_0\in\mathscr C_0(\OO)$. We employ the optimality already known for $\mu_1$ in \rf{JPA}. That is, we verify that if \rf{CT-vi} holds, then \rf{OPT-II} also is true. Hence we arrive at a contradiction. In the case of  \rf{CT-ivb}$_{2,3}$, assume  $q\geq n$ and $p\in[\frac n2,n)$. Then, under assumption \rf{CT-vi}$_1$, recalling \rf{JPA}$_2$, we get \be\label{IPR-I}\ba{ll}\dm \n v(t)\dm_q\leq \mbox{$\xi(\frac t2)t^{\frac12-\frac n{2p} }\dm \n v(\frac t2)\dm_p$}\hskip-0.2cm&\leq c\xi(\frac t2)t^{ -\frac n{2p}-\mu}\dm v_0\dm_\frac n2 \vspace{3pt}\\&=c\xi(\frac t2)t^{ -1}\dm v_0\dm_\frac n2 ,\;t>\max\{2,t_0\}\,,\ea\ee that is \rf{OPT-II}.   Now let us consider the case of $q\in[p,n)$. The argument is similar. Assume that \rf{CT-vi}$_1$ holds for $q\in [p,n)\,,\, p\geq\frac n2$.  Then,    for $r\geq n$, via \rf{CT-ivb}$_3$ and \rf{JPA}$_2$,  we also get 
$$\dm\n v(t)\dm_r\!\leq ct^{\frac 12-\frac n{2q}}\dm \n v(\mbox{$\frac t2)\dm_q \! \leq c\xi(\frac t4)t^{\frac12-\frac n{2p}}\dm \n v(\frac t4 )\dm_p\! \leq \!c\xi(\frac t4$})t^{-1}\dm   v_0 \dm_\frac n2\,,\;t>\!\max\{4,t_0\}\,,$$ that is \rf{IPR-I}, which is false.  Considering the properties \rf{JPA} of $v$ in the case of $n=3$, one achieves the proof of  \rf{CT-ivb}$_4$ by repeating the same arguments. Actually, assuming that \rf{CT-vi}$_2$ holds for some $p\in[\frac n2,n)$, and $v_0\in\mathscr C_0(\OO)$, by virtue of \rf{JPA}$_2$, we get the following estimate:$$\dm\n v(t)\dm_3\leq \xi(\mbox{$\frac t2$})t^{\frac 14-\frac 3{4p}-\frac\theta2}\dm\n v(\mbox{$\frac t2$})\dm_p\leq \xi(\mbox{$\frac t2$})t^{-\frac54+\frac3{4p}-\frac\theta2}\dm v_0\dm_{\frac n2}\leq \xi(\mbox{$\frac t2$})t^{-1+\frac\delta2}\dm v_0\dm_{\frac n2},$$ where in the last step we set $\theta=\frac 32\big[\frac1p-\frac13\big]-\delta>0$\,. So that, by the assumption \rf{CT-vi}$_2$, for all $\delta\in(0,\frac32\big[\frac1p-\frac13\big])$ we find $\ov\xi(t):=\xi(\frac t2)t^{-\frac\delta2}$ such that \rf{OPT-II} holds for $q=3$ and $s=\frac n2$,  that is an {\it absurdum}.
 To prove point ii. of the proposition it is enough to consider a solution ensured by Theorem\,\ref{SSTPT}. Such a solution solves problem \rf{STI} and satisfies estimate \rf{CT-iii}
with $g_p(t)$ constant $\geq1$\,. \chiu
\end{document}